\documentclass[reqno,12pt]{amsart}
\usepackage{amsmath, latexsym, amsfonts, amssymb, amsthm, amscd}
\usepackage{mathrsfs,enumerate}

\numberwithin{equation}{section}

\newenvironment{Proof}{\removelastskip\par\medskip   
\noindent{\em Proof.}
\rm}{\penalty-20\null\hfill$\square$\par\medbreak}

\newtheorem{theorem}{Theorem}[section]
\newtheorem{lemma}[theorem]{Lemma}
\newtheorem{prop}[theorem]{Proposition}

\newtheorem{definition}[theorem]{Definition}

\newcommand{\ga}{\alpha}

\newcommand{\gep}{\varepsilon}       

\newcommand{\cB}{{\ensuremath{\mathcal B}} }
\newcommand{\cF}{{\ensuremath{\mathcal F}} }
\newcommand{\cP}{{\ensuremath{\mathcal P}} }
\newcommand{\cE}{{\ensuremath{\mathcal E}} }

\newcommand{\cK}{{\ensuremath{\mathcal K}} }

\newcommand{\cL}{{\ensuremath{\mathcal L}} }

\newcommand{\cD}{{\ensuremath{\mathcal D}} }

\newcommand{\bbD}{{\ensuremath{\mathbb D}} }
\newcommand{\bbE}{{\ensuremath{\mathbb E}} }
\newcommand{\E}{{\ensuremath{\mathbb E}} }

\newcommand{\bbH}{{\ensuremath{\mathbb H}} }

\newcommand{\N}{{\ensuremath{\mathbb N}} }

\newcommand{\bbP}{{\ensuremath{\mathbb P}} }
\newcommand{\bbQ}{{\ensuremath{\mathbb Q}} }
\newcommand{\bbR}{{\ensuremath{\mathbb R}} }
\newcommand{\R}{{\ensuremath{\mathbb R}} }
\newcommand{\bbS}{{\ensuremath{\mathbb S}} }

\newfont{\indic}{bbmss12}

\def\un#1{\hbox{{\indic 1}$_{#1}$}}

\author{Said Karim Bounebache}
\address{Laboratoire de Probabilit{\'e}s et Mod\`eles Al\'eatoires (CNRS U.M.R. 7599) \\ Universit{\'e} Paris 6
-- Pierre et Marie Curie, U.F.R. Math\'ematiques, Case 188, 4 place
Jussieu, 75252 Paris cedex 05, France }
\email{said.bounebache\@@etu.upmc.fr}

\author{Lorenzo Zambotti}
\address{Laboratoire de Probabilit{\'e}s et Mod\`eles Al\'eatoires (CNRS U.M.R. 7599) \\ Universit{\'e} Paris 6
-- Pierre et Marie Curie, U.F.R. Math\'ematiques, Case 188, 4 place
Jussieu, 75252 Paris cedex 05, France }
\email{lorenzo.zambotti\@@upmc.fr}

\title{A skew stochastic heat equation}



\date{}

\begin{document}

\maketitle

\begin{abstract}
We consider a stochastic heat equation driven by a space-time white noise and with a singular drift, where a local-time in space appears. The process we study has an explicit invariant measure of Gibbs type, with a non-convex potential. We obtain existence of a Markov solution, which is associated with an explicit Dirichlet form. Moreover we study approximations of the stationary solution by means of a regularization of the singular drift or by a finite-dimensional projection. 
\end{abstract}

\section{Introduction}

\subsection{The skew Brownian motion} 
Consider the following stochastic differential equation in $\R$:
\begin{equation}\label{skewR}
X_t = X_0 + B_t + \beta L^0_t, \qquad t\geq 0,
\end{equation}
where $(B_t)_{t\geq 0}$ is a standard Brownian motion in $\R$,
$(L^0_t)_{t\geq 0}$ is the local time at $0$ of the process $(X_t)_{t\geq 0}$, namely
\begin{equation}\label{skewR1}
L^0_t = \lim_{\gep\to 0} \frac1{2\gep} \int_0^t \un{(|X_s|\leq \gep)} \, ds.
\end{equation}
Harrison and Shepp \cite{hash} have proved that equation \eqref{skewR}-\eqref{skewR1} has a unique solution iff $|\beta|\leq 1$ and there is no solution if $|\beta|> 1$. In the former case, the process $(X_t)_{t\geq 0}$ has the law of the {\it skew Brownian motion} with parameter $\alpha=(1+\beta)/2$, i.e. a Brownian motion whose excursions are chosen to be positive, respectively negative, independently of each other, and each with probability $\alpha$, resp. $1-\alpha$.

In this paper we want to introduce a stochastic heat equation which has some analogy with \eqref{skewR}-\eqref{skewR1}. Let us also note that an invariant measure for $(X_t)_{t\geq 0}$ is given by
\[
m_\alpha(dx) = (1-\alpha) \un{(x>0)} \, dx + \alpha \un{(x<0)} \, dx
= C\exp( - c \, \un{(x>0)}(x)) \,dx,
\]
where $c,C$ are constants depending on $\alpha$.
Moreover $(X_t)_{t\geq 0}$ is associated with the Dirichlet form in $L^2(m_\alpha)$
\[
E(u,v) := \frac 12 \int_\R u' v' \, dm_\alpha.
\]

\subsection{A skew SPDE}
In this paper we want to study a {\it skew stochastic heat equation},
namely the stochastic partial diffential equation (SPDE)
\begin{equation}\label{0}
\left\{ \begin{array}{ll} {\displaystyle \frac{\partial u}{\partial
t}=\frac 12\,
\frac{\partial^2 u}{\partial \theta^2} +
\frac\alpha2\frac{\partial}{\partial\theta} \ell^0_\theta + \dot{W}, }
\\ \\
u(t,0)=u(t,1)=0 ,
\\ \\
u(0,\theta)=u_0(\theta), \quad \theta\in[0,1]
\end{array} \right.
\end{equation}
where $(\ell^a_{t,\theta}, \theta\in[0,1])$ is the family of local times at $a\in\R$ accumulated over $[0,\theta]$ by the process $(u(t,r), r\in[0,1])$,
$W(t,\theta)$ is a Brownian sheet over $[0,+\infty[\times[0,1]$ and $\dot{W}(t,\theta)$ is
therefore a space-time white-noise and  $u_0\in L^2(0,1)$. In fact, we consider a more general version of equation \eqref{0}, see \eqref{1} below.

We recall that the {\it stochastic heat equation} is given by
\begin{equation}\label{sheat}
\left\{ \begin{array}{ll}
\displaystyle{ \frac{\partial v}{\partial t} = \frac 12 \, \frac{\partial^2v}{\partial \theta^2} + \dot{W},}
\\ \\ v(t,0)=v(t,1)=0
\\ \\ v(0,\theta)=u_0(\theta), \qquad x\in[0,1]
\end{array}
\right.
\end{equation}
The process $(v_t, t\geq 0)$ is an-infinite dimensional Ornstein-Uhlenbeck process and it is associated with the Dirichlet form
\[
{\cE}^0(\varphi,\psi) \, := \,
\frac 12 \, \int_H \langle \nabla \varphi ,
\nabla \psi \rangle \, d\mu,
\]
in $L^2(\mu)$, where $H:= L^2(0,1)$, $\nabla$ is the Fr\'echet gradient on $H$ and
$\mu$ is the law of a standard Brownian bridge from 0 to 0 over $[0,1]$, see \cite{dpz2}.

Equation \eqref{0} is naturally associated with a perturbation of ${\cE}^0$, defined by means of the probability measure on $H$
\[
\nu(dx) := \frac 1Z \exp\left( - \alpha\int_0^1 \un{(x_s> 0)} \,ds\right) \, \mu(dx),
\]
with $\alpha\in\R$, and of the Dirichlet form
\begin{equation}\label{nu}
{\cE}(\varphi,\psi) \, := \,
\frac 12 \, \int_H \langle \nabla \varphi ,
\nabla \psi \rangle \, d\nu,
\end{equation}
in $L^2(\nu)$. 
Equation \eqref{0} is therefore a natural infinite-dimensional version of \eqref{skewR}: indeed, its invariant measure $\nu$ favors paths over $[0,1]$ which spend more time in the positive axis than in the negative one. The definition and construction of this process are non-trivial, for several reasons.

First, the local-time term plays the role of a very singular drift, which furthermore lacks any dissipativity property; this makes a well-posedness result difficult to expect. Secondly, the explicit invariant  measure $\nu$ is not {\it log-concave}, a condition which would insure a number of nice properties of the Dirichlet form $\cE$ and of the associated Markov process, see e.g. \cite{asz} and section \ref{logc} below.

In particular, the process is not Strong-Feller, or at least a proof of this property is out of our reach, see \cite{sc} for a host of examples and consequences of this nice continuity property. We are at least able to prove something weaker, namely the {\it absolute continuity} of the transition semigroup w.r.t. the invariant measure $\nu$, see Proposition \ref{r_lambdaabs} below; our proof of this technical step seems to be new and of independent interest.

\medskip
We also consider two different regularizations of equation \eqref{1}: first we approximate $f$ with a sequence of smooth functions; then we consider finite-dimensional projections (without regularizing $f$). In both cases we prove convergence in law of the associated stationary processes. The main technical tool is the $\Gamma$-convergence (or, in this context, the {\it Mosco-convergence}) of a sequence of Dirichlet forms with underlying Hilbert space depending on $n$. This notion has been introduced by Kuwae and Shioya in \cite{kuwshi} as a generalization of the original idea of Mosco \cite{mosco} and later developed by Kolesnikov in \cite{koles} for finite-dimensional and a particular class of infinite-dimensional problems. Our approach has been largely inspired by the recent work of Andres and von Renesse, see \cite{avr,avr2}.

\subsection{Main results}

We start by giving the main definition. We consider a bounded function $f:\R\mapsto\R$ with bounded variation and we want to study the following
equation
\begin{equation}\label{1}
\left\{ \begin{array}{ll} {\displaystyle \frac{\partial u}{\partial
t}=\frac 12\,
\frac{\partial^2 u}{\partial \theta^2} -\frac12\int_\R f(da)
\frac{\partial}{\partial\theta} \ell^a_{t,\theta} + \dot{W}, }
\\ \\
u(t,0)=u(t,1)=0 ,
\\ \\
u(0,\theta)=u_0(\theta), \quad \theta\in[0,1]
\end{array} \right.
\end{equation}
where $(\ell^a_{t,\theta}, \theta\in[0,1])$ is the family of local times at $a\in\R$ accumulated over $[0,\theta]$ by the process $(u(t,r), r\in[0,1])$.
\begin{definition}
Let $x \in L^2(0,1)$. An adapted process $u$, defined on a complete filtered probability space $(\Omega, \cF, (\cF_t)_t,\bbP)$, is a weak solution of \eqref{1} if 
\begin{itemize}
\item a.s. $u\in C(]0,T]\times[0,1])$ and $\E[\|u_t-x\|^2]\to 0$
as $t\downarrow 0$
%
\item a.s. for $dt$-a.e. $t$ the process $(u(t,r), r\in[0,1])$ has a family of local times $[0,1]\times\R\ni(r,t)\mapsto\ell^a_{t,\theta}$, $a\in\R$, such that
\[
\int_0^\theta g(u(t,r)) \, dr = \int_\R g(a)\, \ell^a_{t,\theta} \, da, \qquad
\theta\in[0,1], \ t\geq 0,
\]
for all bounded Borel $g:\R\mapsto\R$.
\item there is a Brownian sheet $W$  such that
 for all $h \in C^2_c((0,1))$ and $0<\gep\leq t$
\begin{equation}\label{wweak}
\begin{split}
\langle u_t-u_\gep, h\rangle = & \ 
\frac12\int_\gep^t \langle h'', u_s \rangle_{L^2(0,1)} \,ds +  \frac12\int_\gep^t \int_\R f(da) \int_0^1 h'(\theta) \, \ell^a_{s,\theta} \, d\theta \, ds
 \\
  & + \int_\gep^t \int_0^1 h(\theta)\,W(ds,d\theta)
\end{split}
\end{equation}
\end{itemize}
\end{definition}

A Brownian sheet is a Gaussian process $W=\{W(t,\theta):(t,\theta)\in \R_+^2\}$
defined on $(\Omega,{\cF}, {\mathbb P})$, such that $\{W(t,\theta): \, \theta\in \R_+\}$ is $\cF_t$-measurable
for all $t\geq 0$,
with zero mean and covariance function
\[
{\mathbb E}\, [W(t,\theta)W(t',\theta')]=(t\wedge t')(\theta\wedge\theta'),
\qquad t,\theta,t',\theta'\in{\R_+}.
\]
In section \ref{secDir} we study the Dirichlet form $\cE$ defined by \eqref{nu}, proving in particular that it satisfies the absolute continuity condition, namely the resolvent operators have kernels which admit a density with respect to the reference measure $\nu$. In section \ref{secweak} we show that the Markov process associated with $\cE$ is a weak solution of \eqref{1}. Altough for general $f$ a uniqueness result for solutions to \eqref{1} seems to be out of reach, the process we construct is somewhat canonical, since it is associated with the Dirichlet form $\cE$ and moreover it is obtained as the limit of natural regularization/discretization procedures, as shown in sections \ref{ida}, respectively 
\ref{sec:fda}. Indeed, in section \ref{ida} we regularize the nonlinearity $f$ and show that the (stationary) solutions to the approximated equations converge to the stationary solution of  \eqref{1}. In section \ref{sec:fda} we show convergence of finite-dimensional processes, obtained via a space-discretization, to the solution of  \eqref{1}.

\subsection{Motivations}
There is an extensive literature on reaction-diffusion stochastic partial differential equations of the form
\[
\frac{\partial u}{\partial
t}=\frac 12\,
\frac{\partial^2 u}{\partial \theta^2} -\frac12\,f'(u) + \dot{W}, \qquad t\geq 0, \, \theta\in[0,1],
\]
see for instance the monography by Cerrai \cite{sc}; note that by the occupation times formula, 
for smooth $f$ this equation is equivalent to \eqref{1}. This kind of equation has also been used as a
model for fluctuations of effective interface models, see \cite{funaki}. 
However, in order to give a sense
to the above equation, it is typically assumed that $f$ is smooth or convex. In this paper we study 
this equation in the case where $f$ is neither convex nor necessarily smooth and can even have jumps.

One of the motivations of this work is given by the problem of extending the results of \cite{fuol} on 
convergence of fluctuations of a stochastic interface near a hard wall to a non log-concave situation. In particular,
it is a long standing problem to prove the same result as in \cite{fuol} for a critical pinning model, see e.g. \cite{dgz}, where
the invariant measure converges in the limit to the law of a reflecting Brownian motion. Such a situation
is highly non log-convex and the techniques developed for instance in \cite{asz} do not apply. In this paper we show that the $\Gamma$-convergence is an effective tool also in this context.

\subsection{Notations}

We consider the Hilbert space $H=L^2(0,1)$ endowed with the canonical scalar product
\[
\langle h,k \rangle_H \, := \, \int_0^1 h_\theta \, k_\theta \, d\theta,\qquad
\|h\|^2:=\langle h,h\rangle,
\qquad h,k\in H.
\]
\[
C_0:=C_0(0,1):= \{c:[0,1]\mapsto{\mathbb R}\ \
{\rm continuous},\ c(0)=c(1)=0\},
\]
\[
A:D(A)\subset H\mapsto H, \quad D(A):=W^{2,2}\cap W^{1,2}_0(0,1), \quad
A:=\frac 12 \frac{d^2 }{d \theta^2}.
\]
We introduce the following function spaces:
\begin{itemize}
\item We denote by $C_b(H)$ the
space of all $\varphi:H\mapsto{\mathbb R}$ being bounded
and uniformly continuous in the norm of $H$.
We let $\|\varphi\|_\infty:=\sup |\varphi|$. Then $(C_b(H),
\|\cdot\|_\infty)$ is a Banach space.
\item We denote by ${\rm Exp}_A(H) $ the linear span of
$\{1, \cos(\langle \cdot, h\rangle),
\sin(\langle \cdot,h\rangle): h\in D(A) \}$.
\item The space ${\rm Lip}(H)$ is
the set of all $\varphi\in C_b(H)$ such that:
\[
\|\varphi\|_{\rm Lip}\, := \, \|\varphi\|_\infty \,
+ \, \sup_{x\ne y} \frac{|\varphi(x)-\varphi(y)|}{\|x-y\|} \, < \, \infty.
\]
\item The space $C_b^1(H)$ is defined
as the set of all Fr\'echet-differentiable $\varphi\in C_b(H)$,
with continuous bounded gradient $\nabla\varphi:H\mapsto H$.
\end{itemize}
We sometimes write: $m(\varphi)$ for $\int_H\varphi\, dm$,
$\varphi\in C_b(H)$.

\section{The Dirichlet form $\cE$}\label{secDir}

In this section we give a detailed study of the Dirichlet form $\cE$, proving in particular that it
satisfies the {\it absolute continuity property}, see Proposition \ref{r_lambdaabs} below.

\subsection{A non-log-concave probability measure}\label{logc}
Let $\beta=(\beta_\theta, \theta\in[0,1])$ be a standard Brownian bridge and let us denote its law by $\mu$. Then $\mu$
is a Gaussian measure on the Hilbert space $H=L^2(0,1)$.
We consider a bounded function $f:\R\mapsto\R$ with bounded variation and we define $F:H\mapsto \R$:
\[
F(x) := \int_0^1 f(x_r)\, dr, \qquad x\in H.
\]
We define the probability measure on $H$
\begin{equation}\label{mesinv}
\nu(dx) = \frac 1Z \exp(-F(x))\, \mu(dx), \qquad
Z:=\int \exp(-F)\, d\mu.
\end{equation}
where $Z$ is  normalizing
constant. Note that $f$ is not assumed to be convex, and therefore $\nu$ is in general not log-concave, see \cite{asz}. Finally we have clearly
\begin{equation}\label{estcH}
\frac 1C\|\cdot\|^2_{L^2(\mu)}\leq\|\cdot\|^2_{L^2(\nu)}\leq C\|\cdot\|^2_{L^2(\mu)}
\end{equation}
for some constant $C>0$, since $f$ is bounded.

\subsection{The Gaussian Dirichlet Form}\label{cE}
We define now
\[
{\cE}^0(\varphi,\psi) \, := \,
\frac 12 \, \int_H \langle \nabla \varphi ,
\nabla \psi \rangle \, d\mu,
\qquad \forall \varphi,\psi\in C_b^1(H).
\]
Then it is well known that the symmetric positive bilinear form
$({\cE}^0,{\rm Exp}_A(H))$ is
closable in $L^2(\mu)$, see e.g. \cite{dpz3}: we denote
by $({\cE}^0,D({\cE}^0))$
the closure. We recall that $\mu$, law of a standard Brownian bridge $\beta$, has covariance $Q:=(-2A)^{-1}$, a compact operator on $H$ which can be diagonalized as follows:
\[
Qh = \sum_{k=1}^\infty \lambda_k\,\langle h,e_k\rangle_H \, e_k, \qquad h\in H,
\]
where
\[
\lambda_k := \frac1{(\pi k)^2}, \qquad
e_k(x) := \sqrt 2 \, \sin(k\pi x), \qquad x\in[0,1], \ k\in\N^*.
\]
It is well known that the Markov process defined by \eqref{sheat}, i.e. the solution of the stochastic heat equation, is associated with the Dirichlet form $({\cE}^0,D({\cE}^0))$ in $L^2(\mu)$. This process is Gaussian and can be written down explicitly as a stochastic convolution.
We recall the following result from \cite{dpz3}:
\begin{prop}\label{compact} 
Let $\Gamma:=\{\gamma:\N^*\mapsto\N: \sum_k\gamma_k<+\infty\}$.
Then there exists a complete orthonormal system $(H_\gamma)_{\gamma\in\Gamma}$ in $L^2(\mu)$
such that
\[
{\cE}^0(\varphi,\varphi) = \sum_{\gamma\in\Gamma} \Lambda_\gamma \, \langle \varphi, H_\gamma\rangle_{L^2(\mu)}^2, \qquad \forall \varphi\in D({\cE}^0),
\]
where $\Lambda_\gamma$ is given by
\begin{equation}\label{Lambda}
\Lambda_\gamma := \sum_{k\in\N^*} \gamma_k\, \lambda_k^{-1}.
\end{equation}
In particular, the embedding $D({\cE}^0)\mapsto L^2(\mu)$ is compact.
\end{prop}
It follows that $(H_\gamma)_{\gamma\in\Gamma}$ is a c.o.s. of eigenvalues of the Ornstein-Uhlenbeck operator associated with $\cE^0$. We denote by $(P^0_t)_{t\geq 0}$ the associated semigroup in $L^2(\mu)$, which can be of course written as
\[
P^0_t\varphi = \sum_{\gamma\in\Gamma} e^{-\Lambda_\gamma\,t} \, \langle \varphi, H_\gamma\rangle_{L^2(\mu)} \, H_\gamma, \qquad \forall \varphi\in L^2(\mu).
\]
Then we have the following
\begin{prop}\label{prop:hs} For all $t>0$ the operator $P^0_t:L^2(\mu)\mapsto L^2(\mu)$ is Hilbert-Schmidt, i.e.
\begin{equation}\label{hs}
\sum_{\gamma\in\Gamma} e^{-2\Lambda_\gamma\,t}= \prod_{k=1}^\infty \frac1{1-e^{-2t\pi^2k^2}}<+\infty, \qquad t>0.
\end{equation}
In particular, the series
\[
p^0_t(x,y):= \sum_{\gamma\in\Gamma} e^{-\Lambda_\gamma\,t} \, H_\gamma(x) \, H_\gamma(y)
\]
converges in $L^2(\mu\otimes\mu)$ and yields an integral representation of $P^0_t$:
\[
P_t^0 \varphi(x) = \int \varphi(y) \, p_t^0(x,y) \, \mu(dy), \qquad
\mu{\rm -a.e.} \ x, \ \forall \, \varphi\in L^2(\mu).
\]
\end{prop}
\begin{proof}
Let us define $C_n$, for $n\in\N$, as the number of $\gamma\in\Gamma$ such that
$\sum_k \gamma_k \, k^2 = n$.
Then
\[
\sum_{\gamma\in\Gamma} e^{-2\Lambda_\gamma\,t}= \sum_{\gamma\in\Gamma} \sum_{n=0}^\infty \un{(\Lambda_\gamma=n)} \, e^{-2\Lambda_\gamma\,t}= \sum_{n=0}^\infty
C_n e^{-2\pi^2t\,n}.
\]
Now, by a classical formula due to Euler, the generating function of the sequence $(C_n)_{n\geq 0}$ is given by
\[
\chi(r) := \sum_{n=0}^\infty C_n r^n = \prod_{k=1}^\infty \frac1{1-r^{k^2}}, \qquad |r|<1.
\]
The infinite product converges, since by taking the logarithm
\[
-\log \left(1-r^{k^2}\right) \sim r^{k^2}, \qquad k\to+\infty,\qquad |r|<1,
\]
which is a summable sequence. By choosing $r=e^{-2t\pi^2}$, the first claim follows.
The rest is a trivial consequence of this result.
\end{proof}

From \eqref{hs} one can obtain the following
\begin{prop}\label{comp}
The embedding $D({\cE}^0)\mapsto L^2(\mu)$ is not Hilbert-Schmidt.
\end{prop}
\begin{proof} The embedding $D({\cE}^0)\mapsto L^2(\mu)$ is Hilbert-Schmidt if and only if
\[
\sum_{\gamma\in\Gamma\setminus\{0\}} \frac1{\Lambda_\gamma} < +\infty.
\]
Again we can write
\[
\sum_{\gamma\in\Gamma\setminus\{0\}} \frac1{\Lambda_\gamma}= \sum_{\gamma\in\Gamma} \sum_{n=1}^\infty \un{(\Lambda_\gamma=n)} \, \frac1{\Lambda_\gamma}= \sum_{n=1}^\infty
 \frac{C_n}n.
\]
Now, using the generating function $\chi$ of the sequence $C_n$ we obtain
\[
\sum_{n=1}^\infty
 \frac{C_n}n = \int_0^1 dr \sum_{n=1}^\infty {C_n} r^{n-1} = \int_0^1
 \frac{\chi(r)-1}r \,  dr,
\]
since $C_0=1$. The latter integral converges near $0$, but it diverges near $1$, since $\chi(r)\geq (1-r)^{-1}$.
Therefore the above sum is infinite.
\end{proof}

\subsection{The Dirichlet form associated with \eqref{1}}
We define the symmetric positive bilinear form
\[
{\cE}(\varphi,\psi) \, := \,
\frac 12 \, \int_H \langle \nabla \varphi ,
\nabla \psi \rangle \, d\nu,
\qquad \forall \, \varphi,\psi\in C_b^1(H).
\]
Let us set $\cK:={\rm Exp}_A(H)$.
\begin{lemma}\label{lem2}
The symmetric positive bilinear form $({\cE},\cK)$ is
closable in $L^2(\nu)$. We denote
by  $({\cE},D({\cE}))$ the closure.
\end{lemma}
\begin{proof}
By \eqref{estcH} we have that
\begin{equation}\label{estcE}
\frac 1C\,\cE^0_1\leq \cE_1\leq C\,\cE^0_1.
\end{equation}
Closability of $({\cE}^0,\cK)$ yields immediately the result.
\end{proof}

\subsection{Absolute continuity}

Let $(P_t)_{t\geq 0}$ be the semigroup associated with the Dirichlet form $(\cE,D(\cE))$ in
$L^2(\nu)$. We denote by $R_\lambda:=\int_0^\infty e^{-\lambda\, t}\, P_t\, dt$, $\lambda>0$, the resolvent family of $(P_t)_{t\geq 0}$. In this section we want to prove the following

\begin{prop}\label{r_lambdaabs}
There exists a measurable kernel $(\rho_\lambda(x,dy), \lambda> 0, x\in H)$ such that
\[
R_\lambda\varphi(x) = \int \varphi(y) \, \rho_\lambda(x,dy), \qquad
\nu{\rm -a.e.} \  x, \ \forall \, \varphi\in L^2(\nu),
 \]
and such that
for all $\lambda>0$ and for all $x\in H$ we have $\rho_\lambda(x,dy) \ll \nu(dy)$.
\end{prop}

We are going to use the following result, see \cite[pp. 1543]{dunsch}.
\begin{theorem}[Minimax principle]\label{minmax}
Let $(T,\cD(T))$ a self-adjoint linear operator on the separable Hilbert space $\bbH$ such that $T\geq 0$ and $(\lambda-T)^{-1}$ is a compact operator for some $\lambda>0$. We denote by $\mathcal{S}^n$ the family of $n$-dimensional subspace of $\bbH$, and for $n\geq 1$ we let $\lambda_n$ the number defined as follows
\begin{equation}\label{minmaxeq}
\lambda_n:=\sup_{G\in\mathcal{S}^n }\inf_{u\in (G\cap D(T))\setminus\{0\}} \frac{\langle u,Tu\rangle_\bbH}{\langle u,u\rangle_\bbH}.
\end{equation}
Then there exists a complete orthonormal system $(\psi_n)_{n\geq 1}$ such that
\[
T\, \psi_n = \lambda_n \, \psi_n, \qquad n\geq 1.
\]
In other words, the  sequence $(\lambda_n)_{n\geq 1}$ is the non-decreasing enumeration of the eigenvalues of $T$, each repeated a number of times equal to its multiplicity. Moreover the $\sup$ in \eqref{minmaxeq} is attained for $G$ equal to the span of
$\{\psi_1,\ldots,\psi_n\}$.
\end{theorem}

With the help of Theorem \ref{minmax}, we can first prove the following
\begin{prop}\label{p_ths}
The operator $P_t:L^2(\nu)\mapsto L^2(\nu)$ is Hilbert-Schmidt and there exists a function
$p_t\in L^2(\nu\otimes\nu)$ such that
\[
P_t \varphi(x) = \int \varphi(y) \, p_t(x,y) \, \nu(dy), \qquad
\nu{\rm -a.e.} \  x, \ \forall \, \varphi\in L^2(\nu).
\]
\end{prop}
\begin{proof}
We recall that an analogous result has been proved in Proposition \ref{prop:hs} for the semigroup $(P_t^0)_{t\geq 0}$ associated with the Dirichlet form $(\cE^0,D(\cE^0))$ in $L^2(\mu)$. Now we want to deduce the same result for $(P_t)_{t\geq 0}$.

We apply first Theorem \ref{minmax} to the Ornstein-Uhlenbeck operator $L^0$ associated with $(\cE^0, D(\cE^0))$ in $L^2(\mu)$. Since $R^0_1:=(1-L^0)^{-1}$ maps $L^2(\mu)$ into $D(\cE^0)$ and the embedding $D(\cE^0)\mapsto L^2(\mu)$ is compact by Proposition \ref{comp}, then $R^0_1$ is compact and also symmetric since $\cE^0$ is symmetric. By Proposition \ref{Lambda}, the spectrum of $(-L^0)$ is pure point, its eigenvalues are $(\Lambda_\gamma)_{\gamma\in\Gamma}$ and the associated eigenvectors are the c.o.s. $(H_\gamma)_{\gamma\in\Gamma}$. If we call $(\delta_n^0)_{n\geq 1}$ the non-decreasing enumeration of $(\Lambda_\gamma)_{\gamma\in\Gamma}$, then by the above result we obtain that
\[
\delta_n^0:=\sup_{G\in\mathcal{S}^n }\inf_{u\in (G\cap D(L^0))\setminus\{0\}} \frac{\cE^0(u,u)}{\quad\langle u,u\rangle_{L^2(\mu)}}.
\]
In fact, since the $\sup$ above is attained for $G$ equal to the span of
$\{\psi_1,\ldots,\psi_n\}\subseteq D(\cE^0)$, then we can also write
\[
\delta_n^0=\sup_{G\in\mathcal{S}^n }\inf_{u\in (G\cap D(\cE^0))\setminus\{0\}} \frac{\cE^0(u,u)}{\quad\langle u,u\rangle_{L^2(\mu)}}.
\]
In the same way, setting
\[
\delta_n:=\sup_{G\in\mathcal{S}^n }\inf_{u\in (G\cap D(\cE))\setminus\{0\}} \frac{\cE(u,u)}{\quad\langle u,u\rangle_{L^2(\nu)}},
\]
then $(\delta_n)_{n\geq 1}$ is the non-decreasing enumeration of the eigenvalues of $(-L):D(L)\subset L^2(\nu)\mapsto L^2(\nu)$.
Now, by \eqref{estcH} and \eqref{estcE}, we obtain that
\[
\frac 1C\delta_n^0\leq\delta_n\leq C\delta_n^0, \qquad n\geq 1.
\]
Therefore for $t>0$
\[
\sum_n e^{-2t\delta_n}\leq \sum_n e^{-2t \frac 1C\delta_n^0}
\]
and the latter sum is finite by \eqref{hs}. Therefore $P_t:L^2(\nu)\mapsto L^2(\nu)$ is Hilbert-Schmidt, symmetric
and non-negative. Then Proposition \ref{p_ths} follows from a well-known characterization of operators with such properties.
\end{proof}

\begin{proof}[Proof of Proposition \ref{r_lambdaabs}] In \cite[Theorem 7.2.1]{fot} it is proved that there exist a set of zero capacity $N$ and a measurable Markov kernel $(p_t(x,dy), t\geq 0, x\in N^c)$ on $N^c$, such that
the function $x\mapsto \int \varphi(y) \, p_t(x,dy)$ is $\nu$-a.s. equal to $P_t\varphi$ and quasi-continuous on $N^c$ for all $t,> 0$. By quasi-continuity we want to say that there is a sequence of nondecreasing closed set $(F_n)_n$, with no isolated point, such that the previous map, restricted on $F_n$, is continuous for all $t> 0$ and $N^c = \cup_n F_n$. By Proposition \ref{p_ths}, for $\nu$-a.e. $x$ we have $p_t(x,dy) = p_t(x,y) \, \nu(dy)$, with $p_t\in L^2(\nu\otimes\nu)$ and $p_t\geq 0$, $\nu\otimes\nu$-almost surely. It follows that the kernel $\rho_\lambda(x,dy)$
representing the resolvent operator $R_\lambda:=\int_0^\infty e^{-\lambda\, t}\, P_t\, dt$ is in fact given for $\nu$-a.e. $x$ by $\rho_\lambda(x,dy)=\rho_\lambda(x,y)\, \nu(dy)$, where for $\nu\otimes\nu$-a.e. $(x,y)$
\[
\rho_\lambda(x,y) := \int_0^{+\infty} e^{-\lambda t} \, p_t(x,y) \, dt.
\]
Moreover $R_\lambda\varphi$ is continuous on $N^c$ for all $\varphi\in L^2(\nu)$.
This allows to prove that $\rho_\lambda(x,dy)\ll \nu(dy)$ for all $x\in N$: indeed, if $B$ is a measurable set such that $\nu(B)=0$, then $\rho_\lambda(x,B)=0$ for $\nu$-a.e. $x$ and therefore, by density and continuity, for all $x\in N^c$.
As in \cite{fot}, we can set $\rho_\lambda(x,dy)\equiv 0$ for all $x\in N$, and the proof is complete.
\end{proof}

\section{Existence of a solution}\label{secweak}
In this section we want to prove the following 
\begin{prop}\label{exmarkov}
The Dirichlet form $(\cE,D(\cE))$ is quasi-regular and the associated Markov process is a weak solution of equation \eqref{1}.
\end{prop}
We recall here the basics of potential theory which are needed in what follows,
referring to \cite{fot} and \cite{maro} for all proofs. By Proposition \ref{exmarkov},
the Dirichlet form $(\cE,D(\cE))$ 
is {\it quasi-regular}, i.e. by \cite[Theorem IV.5.1]{maro} it can be embedded into a regular Dirichlet form;
in particular, the classical theory of \cite{fot} can be applied.
Moreover, the important {\it absolute continuity condition} of Proposition \ref{r_lambdaabs}
allows to pass from the stationary solution to quasi-every initial condition: see for
instance \cite[Theorem 4.1.2 and formula (4.2.9)]{fot}.

We denote by $\cF_\infty^\lambda$ (resp.  $\cF_t^\lambda$) the completion of $\cF_\infty^0$ (resp.
completion of $\cF_t^0$ in $\cF_\infty^\lambda$) with
respect to ${\mathbb P}_\lambda$ and we set $\cF_\infty:=\cap_{\lambda\in{\cP}(K)} \, \cF_\infty^\lambda$, $\cF_t:= \cap_{\lambda\in{\cP}(K)} \, \cF_t^\lambda$,
where ${\cP}(K)$ is the set of all Borel probability measures on $K$.

\subsubsection{Capacity and Additive functionals}
Let $A$ be an open subset of $H$, we define by $\cL_A:=\{u \in D(\cE): u \geq 1$, $\nu$-a.e. on $A\}$. Then we set
\[
{\rm Cap}(A)=\left\{ \begin{array}{ll}
\inf\limits_{u\in\cL_A}\cE_1(u,u), & \cL_A \neq \emptyset, \\
+\infty & \cL_A = \emptyset,
\end{array} \right.
\]
where $\cE_1$ is the inner product on $D(\cE)$ defines as follow
\[
\cE_1(u,v) = \cE(u,v)+\int_H u(x)\, v(x) \, d\nu, \quad u,v\in D(\cE).
\]
For any set $A\subset H$ we let
\[
{\rm Cap}(A) = \inf\limits_{B \ {\rm open}, A \subset B\subset H} {\rm Cap}(B)
\]
A set $N\subset H$ is {\it exceptional} if ${\rm Cap}(N)=0$.

By a Continuous Additive Functional (CAF) of $X$, we mean a
family of functions $A_t:E\mapsto {\mathbb R}^+$, $t\geq 0$, such
that:
\begin{itemize}
\item[(A.1)] $(A_t)_{t\geq 0}$ is $({\cF}_t)_{t\geq 0}$-adapted
\item[(A.2)] There exists a set $\Lambda\in{\cF}_\infty$ and
a set $N\subset H$ with ${\rm Cap}(N)=0$ such that
${\mathbb P}_x(\Lambda)=1$ for all $x\in H\setminus N$,
$\theta_t(\Lambda)\subseteq \Lambda$ for all $t\geq 0$, and for all
$\omega\in \Lambda$: $t\mapsto A_t(\omega)$ is
continuous, $A_0(\omega)= 0$ and for all $t,s\geq 0$:
\[
A_{t+s}(\omega) \, = \, A_s(\omega)+A_t(\theta_s\omega),
\]
where $(\theta_s)_{s\geq 0}$ is the time-translation semigroup on
$E$.
\end{itemize}
Moreover, by a Positive Continuous Additive Functional (PCAF) of $X$ we mean a CAF of $X$
such that:
\begin{itemize}
\item[(A.3)] For all $\omega\in \Lambda$: $t\mapsto A_t(\omega)$ is
non-decreasing.
\end{itemize}
Two CAFs
$A^1$ and $A^2$ are said to be equivalent if
\[
{\mathbb P}_x \left( A^1_t=A^2_t \right) \, = \, 1,
\quad \forall t>0, \ \forall x\in K\setminus N.
\]
If $A$ is a linear combination of PCAFs of
$X$, the Revuz measure of $A$
is a Borel signed measure $\Sigma$ on $K$ such that:
\[
\int_{H} \varphi\, d\Sigma \, = \,
\int_{H} {\mathbb E}_x\left[
\int_0^1 \varphi(X_t)\, dA_t \right]\,
\nu(dx), \quad \forall \varphi\in C_b(H).
\]
From theorem VI.2.4 of \cite{maro}, the correspondence between the PCAF and its Revuz measure is one-to-one
\subsubsection{The Fukushima decomposition}
%
Let $h\in C^2_0((0,1);\R^d)$, and set $U:H\mapsto {\mathbb R}$,
$U(x):=\langle x,h \rangle$. By Theorem \ref{exmarkov}, the Dirichlet
Form $({\cE},D({\cE}))$ is quasi-regular. Therefore
we can apply the Fukushima decomposition, as it is stated in
Theorem VI.2.5 in \cite{maro}, p. 180:
for any $U\in {\rm Lip}(H)\subset D({\cE})$,
we have that there exist an exceptional set $N$, a Martingale Additive Functional of finite
energy $M^{[U]}$ and a Continuous Additive Functional of zero energy
$N^{[U]}$, such that for all $x\in K\setminus N$:
\begin{equation}\label{tipler}
U(X_t)-U(X_0) \, = \, M^{[U]}_t + N^{[U]}_t, \quad t\geq 0, \
{\mathbb P}_x-{\rm a.s.}
\end{equation}

\subsubsection{Smooth measures}

We recall now the notion of smoothness for a positive Borel measure $\Sigma$ on $H$,
see \cite[page 80]{fot}. A positive Borel measure $\Sigma$ is {\it smooth} if
\begin{enumerate}
\item $\Sigma$ charges no set of zero capacity
\item there exists an increasing sequence of closed sets $\{F_k\}$ such that
$\Sigma (F_n) < \infty$, for all $n$
and $\lim\limits_{n\to \infty}{\rm Cap}(K-F_n)=0$ for all compact set $K$.
\end{enumerate}
By definition, a signed measure $\Sigma$ on $H$ is smooth if its total variation measure $|\Sigma|$ is smooth. That happens if and only if $\Sigma=\Sigma^1 - \Sigma^2$, where $\Sigma^1$ and $\Sigma^2$ are positive smooth measures, obtained from $\Sigma$ by applying the Jordan decomposition, see \cite[page 221]{fot}.

We recall a definition from \cite[Section 2.2]{fot}.
We say that a positive Radon measure $\Sigma$ on $H$ is {\it of finite energy}
if for some constant $C>0$
\begin{equation}\label{finiteenergy}
\int |v|\, d\Sigma\leq C\sqrt{\cE_1(v,v)},\qquad \forall \, v\in D(\cE)\cap C_b(H).
\end{equation}
If \eqref{finiteenergy} holds, then there exists an element $U_1\Sigma$ such that
\[
\cE_1(U_1\Sigma,v) = \int_{H} v\, d\Sigma, \qquad  \forall \, v\in D(\cE)\cap C_b(H).
\]
Moreover, by \cite[Lemma 2.2.3]{fot}, all measures of  finite energy are smooth.

Finally, by \cite[Theorem 5.1.4]{fot}, if $\Sigma$ is a positive smooth measure, then
there exists a PCAF $(A_t)_{t\geq 0}$, unique up to equivalence, with Revuz measure equal to $\Sigma$.

\subsection{The associated Markov process}

We have first the following
\begin{lemma}\label{qr}
The Dirichlet form $(\cE,D(\cE))$ is quasi-regular.
\end{lemma}
\begin{proof}
By \eqref{estcE} and by \cite[Definition IV.3.1]{maro}, 	quasi-regularity of
$(\cE,D(\cD))$ follows from quasi-regularity of
$(\cE^0,D(\cD^0))$, which in turns follows from the fact that this Dirichlet form is
associated with the solution to the stochastic heat equation \eqref{sheat}.
\end{proof}

By \cite[Theorem IV.3.5]{maro}, quasi-regularity implies existence of a Markov process associated with $(\cE,D(\cE))$.

\subsubsection{Existence of local times}
\begin{prop}
Almost surely, for a.e. $t$ there exists a bi-continuous family of local times $[0,1]\ni(r,a)\mapsto\ell^a_{t,r}$ of $(u_t(\theta), \theta\in[0,1])$. 
\end{prop}
\begin{proof}
Let us recall that $\nu$ is equivalent to the law $\mu$ of $(\beta_r, r\in[0,1])$, where $\beta$ is a Brownian bridge over $[0,1]$. Since $\beta$ is a semi-martingale, for $\mu$-a.e. $x$ there exists a family of local times $\ell^a_{r}$ such that
\[
\int_0^r g(x_s) \, ds = \int_\R g(a) \, \ell^a_{r} \, da, \qquad r\in[0,1],
\]
and the map $[0,1]\times\R\ni(r,a)\mapsto\ell^a_{r}\in\R$ is continuous. In particular, setting
\[
S:=\{w\in C([0,1]) : w \text{ has a bi-continuous family of local times $(\ell^a_{r})_{(r,a)\in[0,1]\times\R}$}\},
\]
then $\nu(S)=1$ and therefore
\[
\E_x\left[\int_0^t \un{(u_s\in S^c)} \, ds \right] =
\int_0^t \bbP_x(u_s\in S^c) \, ds = \int_0^t p_s(x,S^c) \, ds =0
\]
since the law of $(u_s(\theta), \theta\in[0,1])$ by Proposition \ref{p_ths} is absolutely continuous w.r.t. $\nu$. Therefore, the time spent by $(u_s, s\geq 0)$ in $S^c$ is a.s. equal to 0.
\end{proof}

We need now an integration by parts formula on the Dirichlet form $\cE$.
We recall the definitions
\[
F(x) := \int_0^1 f(x_r)\, dr, \qquad \rho(x) := \exp(-F(x)),
\qquad x\in H,
\]
where $f:\R\mapsto\R$ is a bounded function with bounded variation.

\begin{prop}\label{prop:ibpf}
For any $h\in D(A)$ and $\varphi\in C^1_b(H)$
\begin{equation}\label{ibpf1}
\E[\rho(\beta)\, \partial_h \varphi(\beta)] =
\E\left[\rho(\beta)\, \varphi(\beta) \left(-\langle h'',\beta\rangle +\int_{\R\times[0,1]} f(da)\, h_r \, \ell^a(dr) \right)
\right].
\end{equation}
\end{prop}
\begin{proof}
Let $h\in D(A)$ and $\gep\in\R$, by the occupation time
formula:
\[
\begin{split}
F(\beta+\gep h) & = \int_0^1 f(\beta_r+\gep h_r)\, dr =
\int_\R \int_0^1 f(a+\gep h_r)\, \ell^a(dr)   \, da
\\ & = \int_{\R\times\R\times[0,1]} da\, f(da)\, \ell^a(dr) \, \un{(a\geq s-\gep h_r)}\qquad {\rm a.s.}
\end{split}
\]
where $(\ell^a(r), a\in\R, r\in[0,1])$ is the local times family of $\beta$.
Therefore
\[
\left. \frac{d}{d\gep}\, F(\beta+\gep h) \right|_{\gep=0} =
-\int_{\R\times[0,1]} f(da)\, h_r \, \ell^a(dr).
\]
Then by using the Cameron-Martin formula
\[
\begin{split}
\E[\rho(\beta)\, \varphi(\beta+\gep h)]  & =
\E[\rho(\beta-\gep h)\, \varphi(\beta) \, \exp\left(-\gep \langle h'',\beta\rangle -
\|h\|^2 \gep^2/2\right)]
\end{split}
\]
and by differentiating w.r.t. $\gep$ at $\gep=0$ we conclude.
\end{proof}

We want now to show that the process associated with $\cE$ satisfies \eqref{1}.
We are going to apply \eqref{tipler} to $U^h(x):=\langle x,h\rangle$, $x\in H$,
with $h\in C^2_c((0,1);\R^d)$. Clearly $U^h\in {\rm Lip}(H)\subset D({\cE})$.
Our aim is to prove the following
\begin{prop}\label{decomposition}
There is an exceptional set $N$ such that
for all $x\in H\setminus N$, $\bbP_x$-a.s. for all $t\geq 0$
\begin{equation} \label{decomposition2}
N_t^{[U^h]}  =  \frac12\int_0^t
 \langle h'', u_s\rangle  \, ds +\frac12\int_{]0,t]\times[0,1]} \int_\R f(da)
\, h_r' \, \ell^a_{s,r} \, ds \, dr
\end{equation}
where a.s. for all $s>0$
\[
-\int_{[0,1]} \int_\R  h_r' \, \varphi(a) \, \ell^a_{s,r} \, dr =
\int_0^1 h_r \, \varphi(u_s(r)) \, dr
, \qquad \forall \, \varphi\in C_b(\R).
\]
\end{prop}

\begin{proof}
The main tools of the proof are the integration by parts formula \eqref{ibpf1} and
a number of results from the theory of Dirichlet forms in \cite{fot}.
We start by applying \eqref{tipler} to $U^h(x):=\langle x,h\rangle$, $x\in H$. By approximation and linearity we can assume that $h\in D(A)$, $h''\geq 0$ and therefore $h\geq 0$ as well.
The process $N^{[U^h]}$ is a CAF
of $X$, and its Revuz measure is $\frac 12\,\Sigma^h$, where
\begin{equation}
\Sigma^h(dw):=\left(\langle w,h''\rangle - \int_{\R\times[0,1]} f(da)\, h_r \, d\ell^a_r \right) \nu(dw)
\end{equation}
and $\ell^a_r$ is the bi-continuous family of local times of the Brownian bridge. Remark that we have the estimate
\[
\E\left( \left(\int_{\R\times[0,1]} f(da)\, h_r \, d\ell^a_r \right)^2\right)<+\infty
\]
since $f(da)$ has globally bounded variation, $h$ is bounded and $\ell^a_1$ is in $L^p$ for any $p\geq 1$.

By linearity, it is enough to consider the case $h\geq 0$. Then 
the measurable function $\Phi(w) := \int_{[0,1]\times\bbR} h_r \, d\ell^a_r\, f(da)$ is non-negative, and $\Phi \, d\nu$ is a measure with finite energy, since
\[
\int |v| \, \Phi \, d\nu \leq \|\Phi\|_{L^2(\nu)} \, \|v\|_{L^2(\nu)} \leq \|\Phi\|_{L^2(\nu)} \,\sqrt{\cE_1(v,v)}, \qquad \forall \, v\in D(\cE)\cap C_b(H),
\]
see \eqref{finiteenergy} above. In particular, $\Phi \, d\nu$ is a smooth measure.
By theorem 5.1.3 of \cite{fot}, there is an associated PCAF, denoted by $N_t$. Notice that the process
\[
N^{n}_t:=\int_0^t (\Phi\wedge n)(X_s) \, ds
\]
is a well defined PCAF with Revuz measure $\Phi\wedge n\, d\nu$ and $N^n_t\leq N_t$, since $N_t-N^n_t$ is a CAF with a non-negative Revuz measure. By monotone convergence we find for all non-negative $ \varphi\in C_b(H)$
\[
\begin{split}
\int_{H} \varphi\, \Phi \, d\nu & = \lim_n 
\int_{H} \varphi\, \Phi\wedge n \, d\nu \, = \lim_n
 {\mathbb E}_\nu\left[
\int_0^1 \varphi(X_t)\, (\Phi\wedge n)(X_t) \, dt \right]\,
\\ & =   {\mathbb E}_\nu\left[
\int_0^1 \varphi(X_t)\, \Phi(X_t) \, dt \right].
\end{split}
\]
Therefore, $t\mapsto \int_0^t \Phi(X_s) \, ds$ is a PCAF with Revuz measure $\Phi \, d\nu$ and must
therefore be equivalent to $t\mapsto N_t$.
\end{proof}
\subsection{Identification of the noise term}
We deal now with the identification of $M^{[U^h]}$ with
the integral of $h$ with respect to a space-time white noise.
\begin{prop}\label{decomposition3}
There exists a Brownian sheet $(W(t,\theta), \, t\geq 0, \theta\in[0,1])$,
such that
\begin{equation}
M^{[U^h]}_t = \int_0^t \int_0^1 h_{\theta} \, W(ds,d\theta), \qquad h \in H.
\end{equation}
\end{prop}
\begin{proof}
We recall that, for $U\in D(\cE)$, the process $M^{[U]}$ is a continuous
martingale, whose quadratic variation
$(\langle M^{[U]}\rangle_t)_{t\geq 0}$
is a PCAF of $X$ with Revuz measure $\mu_{\langle M^{[U]}\rangle}$ given by
the formula
\begin{equation}\label{quadratic}
\int f\, d\mu_{\langle M^{[U]}\rangle} = 2\cE(Uf,U)-\cE(U^2,f), \qquad
\forall \, f\in D(\cE)\cap C_b(H),
\end{equation}
see \cite[Theorem 5.2.3]{fot}. Now, if we apply this formula
to $U^h(x)=\langle x,h\rangle$, then we obtain
\[
\int f\, d\mu_{\langle M^{[U^h]}\rangle} = \|h\|^2\int f\,d\nu, \qquad
\forall \, f\in D(\cE)\cap C_b(H).
\]
Therefore, the quadratic variation $\langle M^{[U^h]}\rangle_t$ is equal
to $\|h\|^2t$ for all $t\geq 0$, and, by L\'evy's Theorem, $(M^{[U^h]}\cdot\|h\|^{-1})_{t\geq 0}$ is a Brownian motion. Moreover, the parallelogram law,
if $h_1,h_2\in H$ and $\langle h_1,h_2\rangle=0$, then the
quadratic covariation between $M^{[U^{h_1}]}$ and $M^{[U^{h_2}]}$ is equal to
\[
\langle M^{[U^{h_1}]}, M^{[U^{h_2}]}\rangle_t = t \, \langle h_1,h_2\rangle,\qquad
t\geq 0.
\]
Therefore, $(M^{[U^h]}_t, t\geq 0, h\in H)$ is a Gaussian process
with covariance structure
\[
\E_x\left(M^{[U^{h_1}]}_t \, M^{[U^{h_2}]}_s \right) =
s\wedge t\,  \langle h_1,h_2\rangle.
\]
If we define $W(t,\theta):=M^{[U^{h}]}_t$ with
$h:=1_{[0,\theta]}$, $t\geq 0$, $\theta\in[0,1]$, then $W$ is the
desired Brownian sheet.
\end{proof}
%
\begin{proof}[Proof of Proposition \ref{exmarkov}]
Quasi-regularity has been proved in Lemma \ref{qr}. 
First we apply the Fukushima decomposition \eqref{tipler} to the function $U_h(x):=\langle x,h\rangle$ and identify the terms using propositions \ref{decomposition3} and \ref{decomposition} and the above results. It remains to prove that the process $(X_t)_{t\geq 0}$ satisfies the desired continuity properties. To this aim, we use the result of Lemma \ref{apriori} below. We notice that for any $\eta\in(0,1/2)$ and $p>1$
\[
\begin{split}
\frac1C\int_H \| x \|_{W^{\eta,p}(0,1)} ^p \, \nu(dx) & \leq 
 \E\left( \| \beta\|^p_{W^{\eta,p}(0,1)} \right)   \leq  \E  \left(|\beta_r|^p + \int_0^1\int_0^1 \dfrac{|\beta_s - \beta_t|^p }{|s-t|^{p\eta+1}} \,dt\,ds\right)
\\ & \leq 1+ \int_0^1\int_0^1 {|s-t|^{p(\frac12-\eta)-1}} \,dt\,ds <  +\infty.
\end{split}
\]
Then by Lemma \ref{apriori} and by Kolmogorov's criterion in the Polish space $C^\beta([0,1])$ we obtain that 
under $\bbP_\nu$ the coordinate process has a modification in $C([0,T]\times[0,1])$ for all $T>0$.

Finally, in order to prove continuity of a non-stationary solution, we use the absolute-continuity property of proposition \ref{r_lambdaabs}. Let us consider the set $C:=C([0,1])$ endowed with the uniform topology.
Let $S\subset \,]0,+\infty[$ be countable and satisfying $\varepsilon:=\inf S>0$ and
$\sup S<\infty$, and define $B_S\subset
C^{]0,+\infty[}$ as
$$
B_S:=\left\{\omega\in C^{]0,+\infty[}:\ \text{the restriction of $\omega$ to $S$ is
uniformly continuous}\right\},
$$
then we know that $\bbP_\nu(B_S)=1$, i.e. $\bbP_x(B_S)=1$ for $\nu$-a.e. $x$. For all $x\in N^c$, where $N$ is exceptional, the law of $X_\varepsilon$ under $\bbP_x$ is absolutely continuous w.r.t. $\nu$ for all $\varepsilon>0$.
Then $\bbP_{X_\varepsilon}(B_{S-\varepsilon})=1$, $\bbP_x$-almost surely.
Taking expectations, and using the Markov property, we get $\bbP_x(B_S)=1$. Arguing as in \cite[Lemma~2.1.2]{strook} we obtain that $\bbP_x^*\left(C(]0,+\infty[;C)\right)=1$, where $\bbP_\nu^*$ denotes the outer measure.

\end{proof}

\section{Convergence of regularized equations}\label{ida}

In this section we consider a smooth approximation $f_n$ of $f$ and and we study
convergence in law of $u^n$ to $u$, where
\begin{equation}\label{1n}
\left\{ \begin{array}{ll} {\displaystyle \frac{\partial u^n}{\partial
t}=\frac 12\,
\frac{\partial^2 u^n}{\partial \theta^2} -\frac12 \,f'_n(u^n) + \dot{W}, }
\\ \\
u^n(t,0)=u^n(t,1)=0 ,
\\ \\
u^n(0,\theta)=u_0^n(\theta), \quad \theta\in[0,1].
\end{array} \right.
\end{equation}
By a $\Gamma$-convergence technique, we shall prove convergence in law of the stationary  processes.

Since $f$ is bounded and with bounded variation, then it is continuous outside a countable set $\Delta_f$. Moreover we can find a sequence of smooth functions $f_n:\R\mapsto\R$ such that
\begin{enumerate}
\item $(f_n)_n$ is uniformly bounded
\item $f_n\to f$ as $n\to+\infty$ locally uniformly in $\R\setminus \Delta_f$.
\end{enumerate}
We define the probability measure on $H$
\begin{equation}\label{mesinv2}
\nu_n(dx) = \frac 1{Z_n} \exp(-F_n(x))\, \mu(dx), \qquad Z_n:=\int \exp(-F_n)\, d\mu,
\end{equation}
where $Z_n$ is a normalizing
constant. Again, $\nu_n$ is not necessarily log-concave, see \cite{asz}.
Setting
\[
\rho_0:= 1, \quad
\rho_n := \frac{d\nu_n}{d\mu}, \quad n\geq 1, \qquad \rho := \frac{d\nu}{d\mu},
\]
we find that $0<c\leq\rho_n\leq C<+\infty$ and
$0<c\leq\rho\leq C<+\infty$ on $H$, since $f_n$ and $f$ are bounded for all $n\in\N$.
We have then the simple
\begin{lemma}\label{lem1}
There is a canonical identification between the Hilbert spaces $L^2(\nu)$ and
$L^2(\nu_n)$ for all $n\in\N$ and for positive constants $c,C$
\begin{equation}\label{estcH0}
\frac cC\|\cdot\|^2_{L^2(\nu)}\leq\|\cdot\|^2_{L^2(\nu_n)}\leq \frac Cc\|\cdot\|^2_{L^2(\nu)}.
\end{equation}
\end{lemma}
\begin{proof}
This is obvious since $0<c\leq\rho_n\leq C<+\infty$ and
$0<c\leq\rho\leq C<+\infty$.
\end{proof}
In particular we can consider $L^2(\nu_n)$ as being a copy of $L^2(\nu)$ endowed
with a different norm $\|\cdot\|_{L^2(\nu_n)}$. We shall use this notation below.

\medskip
We define the symmetric positive bilinear form
\[
{\cE}^n(\varphi,\psi) \, := \,
\frac 12 \, \int_H \langle \nabla \varphi ,
\nabla \psi \rangle \, d\nu_n,
\qquad  \forall \varphi,\psi\in C_b^1(H),
\]
Let us set $\cK:={\rm Exp}_A(H)$.
\begin{lemma}\label{lem22}
The symmetric positive bilinear forms $({\cE}^n,\cK)$ is
closable in $L^2(\nu_n)$. We denote
by $({\cE}^n,D({\cE}^n))$ the closure.
\end{lemma}
\begin{proof}
The proof is identical to that of Lemma \ref{lem2}.
\end{proof}
We recall that the Dirichlet form $({\cE}^n,D({\cE}^n))$ is associated with the solution of equation \eqref{1n}, see
e.g. \cite{dpz3}.

\subsection{Convergence of Hilbert spaces}

We recall now the following definition, given by Kuwae and Shioya in \cite{kuwshi}.
\begin{definition}\label{kuwshi}
A sequence of Hilbert spaces $\bbH_n$ converges to a hilbert $\bbH$ if there is a family of linear maps $\{\Phi_n:\bbH \to \bbH_n\}$ such that:
\begin{equation}
\lim\limits_{n \to +\infty} \|\Phi_n(x)\|_{\bbH_n} =  \|x\|_{\bbH} ,\quad x\in \bbH
\end{equation}
A sequence $(x_n)_n$, $x_n\in\bbH_n$, converges \emph{strongly} to a vector $x\in \bbH$ if there exists a sequence $(\tilde{x}_n)_n$ in $\bbH$ such that $\tilde{x}_n\to x$ in $\bbH$ and
\begin{equation}
\lim\limits_{n\to +\infty}\varlimsup\limits_{m\to+\infty}\|\Phi_m(\tilde{x}_n) - x_m\|_{\bbH_m} = 0
\end{equation}
and $(x_n)_n$ converge \emph{weakly} to $x$ if
\begin{equation}
\lim\limits_{n\to +\infty} \langle x_n,z_n\rangle_{\bbH_n} = \langle x, z\rangle_{\bbH}
\end{equation}
for any $z\in \bbH$ and sequence $(z_n)_n$ , $z_n\in\bbH_n$, such that $z_n \to z$ strongly.
\end{definition}

\begin{lemma}$ $\label{topology}
\begin{enumerate}
\item
The sequence of Hilbert spaces $L^2(\nu_n)$ converges to $L^2(\nu)$,
by choosing $\Phi_n$ equal to the natural identification of
equivalence classes in $L^2(\nu_n)$ and $L^2(\nu)$.
\item $u_n\in L^2(\nu_n)$ converges strongly to $u\in L^2(\nu)$ if and only if
$u_n\to u$ in $L^2(\nu)$.
\item $u_n\in L^2(\nu_n)$ converges weakly to $u\in L^2(\nu)$ if and only if
$u_n\to u$ weakly in $L^2(\nu)$.
\end{enumerate}
\end{lemma}
\begin{proof}
\begin{enumerate}
\item We have to prove that for all $x\in L^2(\nu)$ we have $\|x\|_{L^2(\nu_n)} \to \|x\|_{L^2(\nu)}$ as $n\to\infty$. Since $e^{-F_n}/Z_n$ converges a.s. to $e^{-F}/Z$ and it is uniformly bounded, then the result follows by dominated convergence.
\item Let $(u_n)_n$ converges strongly to $u\in L^2(\nu)$ so there is a sequence $(\tilde{u}_n)_n$ in $L^2(\nu)$ tending to $u$ in $L^2(\nu)$ such that:
\begin{equation}\label{strongconv}
\lim\limits_{n}\varlimsup\limits_{m} \|\tilde{u}_n-u_m\|_{L^2(\nu)_m} = 0.
\end{equation}
Then we have:
\[
\qquad \varlimsup\limits_{m}\|u -u_m\|_{L^2(\nu)} \leq  \lim\limits_{n}\|u-\tilde{u}_n\|_{L^2(\nu)}+\frac{C}{c}\lim\limits_{n}\varlimsup\limits_{m}\|u_m - \tilde{u}_n\|_{L^2(\nu_m)} =0,
\]
so that $u_n\to u$ in $L^2(\nu)$. Conversely, if $u_n\to u$ in $L^2(\nu)$ then we can consider $\tilde{u}_n=u$ for all $n\in\N$ and \eqref{strongconv} holds.
\item Let $u_n\in L^2(\nu_n)$ be a sequence which converges weakly to $u\in L^2(\nu)$, i.e.
for all $v\in L^2(\nu)$ and any sequence $v_n\in L^2(\nu_n)$ strongly convergent to $v$
\[
\langle u_n, v_n\rangle_{L^2(\nu_n)} \to \langle u,v\rangle_{L^2(\nu)},\quad n\to +\infty.
\]
Let $v_n:=v\cdot \rho\cdot \rho_n^{-1}$, then by the dominated convergence theorem
$\|v_n-v\|_{L^2(\nu)}\to 0$ and by the previous point $v_n\in L^2(\nu_n)$ converges strongly to $v$. So we have
\[
\langle u_n, v\rangle_{L^2(\nu)} = \langle u_n, v_n\rangle_{L^2(\nu_n)}\nonumber\\
\to  \langle u, v\rangle_{L^2(\nu)},\quad n\to+\infty.
\]
Viceversa, let us suppose that for all $v\in L^2(\nu)$ we have $\langle u_n, v\rangle_{L^2(\nu)}\to \langle u, v\rangle_{L^2(\nu)}$ and let us consider any sequence $v_n\in L^2(\nu_n)$ strongly convergent to $v$. Setting $w_n:=v_n\cdot \rho_n\cdot \rho^{-1}$, by dominated convergence $\|w_n-v\|_{L^2(\nu)}\to 0$ and therefore
$\langle u_n, v_n\rangle_{L^2(\nu_n)}=\langle u_n, w_n\rangle_{L^2(\nu)}\to \langle u, v\rangle_{L^2(\nu)}$ and the proof is finished.
\end{enumerate}
\end{proof}

\subsection{Convergence of Dirichlet Forms}
Now we can give the definition of Mosco-convergence of Dirichlet forms.
This concept is useful for our purposes, since it was proved in \cite{kuwshi} to imply the convergence in a strong sense of the associated resolvents and semigroups.
\begin{definition}
If $\cE^n$ is a quadratic form on $\bbH_n$, then $\cE^n$ Mosco-converges to the quadratic form $\cE$ on $\bbH$  if the two following conditions hold:
\begin{enumerate}
\item[{\it Mosco I}.] For any sequence $x_n\in\bbH_n$, converging weakly to $x\in\bbH$,
\begin{equation}\label{MoscoI}
\cE(x,x) \leq \varliminf\limits_{n\to +\infty} \cE^n(x_n,x_n).
\end{equation}
\item[{\it Mosco II}.] For any $x\in\bbH$, there is a sequence $x_n\in\bbH_n$ converging strongly to $x\in\bbH$ such that
\begin{equation}\label{MoscoII}
\cE(x,x) = \lim\limits_{n\to +\infty}\cE^n(x_n,x_n).
\end{equation}
\end{enumerate}
\end{definition}
We say that a sequence of bounded operarors $(B_n)_n$ on $\bbH_n$, converges strongly to an operator B on $\bbH$, if $\bbH_n\ni B_nu_n\to Bu\in \bbH$ strongly for all sequence $u_n\in\bbH_n$ converging strongly to $u\in \bbH$. Then Kuwae and Shioya have proved in \cite{kuwshi} the following equivalence between Mosco convergence and strong convergence of the associated resolvent operators. 
\begin{theorem}[Kuwae and Shioya \cite{kuwshi}]\label{moscoks}
The Mosco convergence is equivalent to the strong convergence of the associated resolvents.
\end{theorem}

\subsection{Mosco convergence}
\begin{prop}\label{MI}
The Dirichlet form $\cE^n$ on $L^2(\nu_n)$ Mosco-converges to $\cE$ on $L^2(\nu)$.
\end{prop}
\begin{Proof}
The proof of the condition Mosco II is trivial in our case; indeed, for all $x\in D(\cE)$, we set $x_n:=x\in D(\cE^n)$ for all $n\in\N$; by dominated convergence
$\cE(x,x) = \lim\limits_{n} \cE^n (x,x)$. If $x\notin D(\cE)$, then again $x_n:=x\notin 
D(\cE^n)$ satisfies $\cE(x,x) = \lim\limits_{n} \cE^n (x,x)=+\infty$.

Let us prove now condition Mosco I.
We first assume that $u\in\cD(\cE)$. By the integration by parts formula \eqref{ibpf1} we have for any $v\in\cK={\rm Exp}_A(H)$
\[
2\cE(u,v) = - \int_H  u\cdot {\rm Tr}(D^2v)\, d\nu + \int_H u\left(\langle  \cdot,A\nabla v\rangle_H - \int_{\R\times[0,1]} f(da)\, \nabla_r v \, \ell^a(dr)\right) d\nu.
\]
Let $u_n\in L^2(\nu_n)$ a sequence converging weakly to $u$, then we know from Theorem \ref{topology} that $u_n\to u$ weakly in $L^2(\nu)$. By the compactness of the embedding $D(\cE^0)\mapsto L^2(\mu)$ proved in Proposition \ref{comp}, $u_n\to u$ strongly in $L^2(\nu)$. By linearity it is enough to consider $v(x)=\exp(i\langle h,x\rangle_H)$, $h\in D(A)$, $x\in H$. Notice that $\nabla v = i\, v\, h$. Then we can write
\[
\int_{\R\times[0,1]} f(da)\, \nabla_r v(\beta) \, \ell^a(dr) = i\, v(\beta) \,  \int_{\R\times[0,1]} f(da)\, h_r \, \ell^a(dr).
\]
Moreover by the occupation times formula
\[
\langle\nabla v(\beta), f_n'(\beta)\rangle_H = i\, v(\beta) \,  \int_0^1 h_r \, f_n'(\beta_r) \, dr =
i\, v(\beta) \,  \int_{\R\times[0,1]}  h_r \, \ell^a(dr) \, f'_n(a)\, da.
\]
Since $f'_n(a)\, da\rightharpoonup f(da)$, by dominated convergence we obtain
\[
2\cE(u,v) = \lim\limits_{n\to \infty}\left(- \int_H  u^n\cdot {\rm Tr}(D^2v)\,d\nu^n + \int_H u^n(\langle  x, A\nabla v\rangle_H + \langle\nabla v, f_n' \rangle) \, d\nu^n\right).
\]
We can suppose that each $u^n$ is in $\cD(\cE^n)$ (else $\cE^n(u^n,u^n) = +\infty$) so we have for any $v\in\cK\setminus\{0\}$
\[
\varliminf_{n\to+\infty}\Big(\cE^n(u^n,u^n)\Big)^{1/2} \geq \varliminf_{n\to+\infty}\frac{\cE^n(u^n,v)}{\sqrt{\cE^n(v,v)}} = \frac{\cE(u,v)}{\sqrt{\cE(v,v)}}
\]
and by considering the $\sup$ over $v$ we obtain the desired result.

Suppose now that $u\notin \cD(\cE)$ and let $L^2(\nu_n)\ni u^n \to u\in L^2(\nu)$ weakly, then we know from Theorem \ref{topology} that $u_n\to u$ weakly in $L^2(\nu)$. By the compactness of the embedding $D(\cE^0)\mapsto L^2(\mu)$ proved in Proposition \ref{comp}, $u_n\to u$ strongly in $L^2(\nu)$. If
$\liminf\limits_{n\to \infty} \cE^n(u^n,u^n) < +\infty$, then we also have
$\liminf\limits_{n\to \infty} \cE(u^n,u^n) < +\infty$. But since $\cE$ is lower semi-continuous in $L^2(\nu)$, then $\cE(u,u)<+\infty$, which is absurd since we assumed that $u\notin \cD(\cE)$.
\end{Proof}

\subsection{Convergence of stationary solutions}
We denote by $\bbP_{\nu_n}^n$ the law of the stationary solution of \eqref{1n} and by $\bbP_\nu$ the law of the Markov process associated with $\cE$ and started with law $\nu$. We have the following convergence result

\begin{prop}\label{convergence}
The sequence $\bbP_{\nu_n}^n$ converges weakly to $\bbP_\nu$ in $C([0,T]\times[0,1])$.
\end{prop}
\begin{proof}
Let us first prove convergence of finite-dimensional distributions, i.e.
\[
\lim_{n\to+\infty} \E^n_{\nu_n}(f(X_{t_1},\ldots,X_{t_m})) =  \E_{\nu}(f(X_{t_1},\ldots,X_{t_m})),
\]
for all $f\in C((C([0,1])^m)$. The Mosco convergence of the Dirichlet forms $\cE^n$ provides the strong convergence of the semi-group and, by the Markov property, the convergence of the finite dimensional laws. Indeed let $f$ be in $C((C([0,1])^m)$ of the form $f(x_1, ...,x_m) = f_1(x_1)\cdot...\cdot f_m(x_m)$  then 
\[
\begin{split}
& P^n_{t_1}(f_1\cdot P^n_{t_2-t_1}(f_2\cdot ... (f_{m-1}P^n_{t_m - t _{m-1}}f_m)...)) 
\\ & \to 
P_{t_1}(f_1\cdot P_{t_2-t_1}(f_2\cdot ... (f_{m-1}P_{t_m - t _{m-1}}f_m)...)), \qquad {\rm strongly}.
\end{split}
\]
Then by the Markov property
\[
\begin{split}
& \E^n_{\nu_n}(f(X_{t_1},\ldots,X_{t_m})) = \langle 1, P^n_{t_1}(f_1\cdot P^n_{t_2-t_1}(f_2\cdot ... (f_{m-1}P^n_{t_m - t _{m-1}}f_m)...)) \rangle_{H_n} \\ & \to
\langle 1, P_{t_1}(f_1\cdot P_{t_2-t_1}(f_2\cdot ... (f_{m-1}P_{t_m - t _{m-1}}f_m)...))\rangle_H = \E_{\nu}(f(X_{t_1},\ldots,X_{t_m})).
\end{split}
\]

We need now to prove tightness in $C([0,T]\times[0,1])$.
We first recall a result of \cite[Th. 7.2 ch 3]{ek}.
Let $(P,d)$ be a Polish space, and let $(X_\alpha)_\alpha$ be a family of processes with sample paths in $C([0,T];P)$. Then the laws of $(X_\alpha)_\alpha$ are relatively compact if and only if the following two conditions hold:
\begin{enumerate}
\item For every $\eta >0$ and rational $t\in[0,T]$, there is a compact set $\Gamma_\eta^t \subset P$ such that:
\begin{equation}\label{tension0}
\inf\limits_\alpha \bbP\left( X_\alpha \in \Gamma_\eta^t\right) \geq 1-\eta
\end{equation}
\item For every $\eta,\epsilon >0$ and $T>0$, there is $\delta>0$ such that
\begin{equation}\label{tension}
\sup\limits_\alpha \bbP\left( w(X_\alpha, \delta, T)\geq \epsilon \right) \leq \eta
\end{equation}
 \end{enumerate}
where $w(\omega, \delta, T):=\sup\{d(\omega(r),\omega(s)): r,s\in[0,T],
\, |r-s|\leq\delta\}$ is the modulus of continuity in $C([0,T];P)$.

We consider now, as Polish space $(P,d)$, the Banach space
$C^{\theta}([0,1])$. Since $\bbP^n_{\nu_n}$ is stationary, \eqref{tension0} is reduced to
a condition on $\nu_n$. In fact we have 
\[
\left( \int_H \| x\|^p_{W^{\eta,p}(0,1)} \, d\nu_n \right)^{\frac1p}
\, \leq \,  \left(\frac Cc \int_H \| x\|^p_{W^{\eta,p}(0,1)} \,d\mu \right)^{\frac1p}.
\]
Now, since the Brownian bridge $(\beta_r)_{r\in[0,1]}$ is a Gaussian process with covariance function 
$r\wedge s-rs$, then
\[
\begin{split}
\E\left( \| \beta\|^p_{W^{\eta,p}(0,1)} \right) &  \leq  \E  \left(\|\beta\|_p^p + \int_0^1\int_0^1 \dfrac{|\beta_s - \beta_t|^p }{|s-t|^{p\eta+1}} \,dt\,ds\right)
\\ & \leq C_p\left(1+ \int_0^1\int_0^1 {|s-t|^{p(\frac12-\eta)-1}} \,dt\,ds \right)<  +\infty.
\end{split}
\]
For any $\eta<1/2$, $\theta<\eta$ and $p>1/(\eta-\theta)$ we have
by the Sobolev embedding Theorem that $W^{\eta,p}(0,1)\subset C^{\theta}([0,1])$ with continuous embedding, so that
$\sup_n \int_H \| x\|^p_{C^{\theta}([0,1])} \, d\nu_n<\infty$. By Lemma \ref{apriori} below we obtain existence of a constant $K$ independent of $n$ such that
\[
{\mathbb E}^n_{\nu_n}\left[\left\| X_t-
X_s\right\|^p_{C^{\theta}([0,1])} \right]
 \, \leq \, K \, |t-s|^{\xi}, \qquad \forall \, n\geq 1, \, t,s\in[0,T].
\]
By Kolmogorov's criterion, see \cite[Thm. I.2.1]{reyo},
we obtain that a.s. $w(X^{n},\delta, T)\leq C \, \delta^{\frac{1-\xi}{2p}}$, with $C\in L^p$. Therefore by the Markov inequality, if $\epsilon>0$
\[
\bbP \left(w(X^{n},\delta, T)\geq \epsilon \right)  \leq \bbE \left[ {C^{p}} \right]
\, \delta^{\frac{1-\xi}2}\epsilon^{-p},
\]
and \eqref{tension} follows for $\delta$ small enough.
\end{proof}

\section{Convergence of finite dimensional approximations}\label{sec:fda}
From now on we turn our attention to another problem: convergence in law of finite dimensional approximations of equation \eqref{1}. We want to project, in a sense to be made precise, \eqref{1} onto an equation in a finite dimensional subspace of $H:=L^2(0,1)$. To be more precise, we consider the space $H_n$ of functions in $L^2(0,1)$ which are constant on each interval $[(i-1)2^{-n}, i2^{-n}[$, $i=1,\ldots,2^{n}$ and we endow $H_n$ with the scalar product inherited from $H$. 

Notice that $H_n$ is a linear closed subspace of $L^2(0,1)$, so that there exists a unique orthogonal projector $P_n:L^2(0,1)\mapsto H_n$, given explicitly by
\begin{equation}\label{P_n}
P_n x := 2^n\sum_{i=0}^{2^n-1} \un{[i2^{-n}, (i+1)2^{-n}[} \, \langle
\un{[i2^{-n}, (i+1)2^{-n}[}, x\rangle.
\end{equation}
We call $\mu_n$ the law of $P_n\beta$; then $\mu_n$ is a Gaussian law on $H$ with zero mean and non-degenerate covariance
operator $P_nQP_n$, where $Q$ is the covariance operator of $\mu$, which has been studied in detail in section \ref{cE}. In what follows we write
\[
P_nQP_n = (-2A_n)^{-1}, \qquad A_n:H_n\mapsto H_n.
\]
We also define $ \pi_n$ as
\begin{equation}\label{mesurefinie}
 \pi_n(dx) = \frac 1{Z_n} \,\exp(-F_n(x))\, \mu_n(dx)
= \frac 1{Z_n} \exp\left(-\frac1 {2^n}\sum_{i=0}^{2^n-1} f(x(i)) \right)\, \mu_n(dx).
\end{equation}
where $Z_n:=\mu_n(\exp(-F_n))$ is a normalization constant.

Then, a natural approximation of $\cE$ defined on $H_n$ is given by
the following symmetric bilinear non-negative form 
\begin{equation}\label{Lambda_n}
\Lambda_n(u,v):=\frac 12\int \langle \nabla u, \nabla v\rangle_{H_n} \,  d\pi_n, \quad u , v \in C^1(H_n)
\end{equation}
with reference measure $\pi_n$.  Then we have
\begin{equation}\label{E^f}
\Lambda^n(u,v) = \frac 12\int \langle \nabla (u\circ P_n), \nabla (v\circ P_n)\rangle_H \, \frac 1{Z_n} \,\exp(-F_n\circ P_n)\, d\mu,
 \quad u , v \in C^1(H_n).
 \end{equation}
We write
\begin{equation}\label{f}
f(y) = f_0(y) + \sum_{j=1}^k \alpha_j \, \un{(y\leq y_j)}, \qquad y\in\R
\end{equation}
where $f_0$ is smooth and bounded and $\alpha_j,y_j\in\R$. Clearly, $f$ has a jump in each $y_j$ of respective size $\alpha_j$. We have the following integration by parts formula
\begin{equation}\label{ibpfn}
\begin{split}
\int \partial_h \varphi \, d \pi_n & = -\int \varphi \, \langle x,A_n h\rangle \,  \pi_n(dx)
 +  \int \varphi(x) \, 2^{-n} \sum_{i=0}^{2^n-1} h_i \, f'_0(x(i)) \, \pi_n(dx)
\\ & 
- \int \varphi(x) \sum_{i=0}^{2^n-1} h_i \, \sum_j 2\, \frac{1-e^{-\alpha_j\, 2^{-n}} }
{1+e^{-\alpha_j \, 2^{-n}}} \,  \pi_n(dx \, ; \, x(i)=y_j),
\end{split}
\end{equation}
where we use the notation
\[
 \pi_n(A \, ; \, x(i)=y_j) := \lim_{\varepsilon\downarrow 0} \frac1{2\varepsilon} \,
 \pi_n(A\cap\{|x(i)-y_j|\leq \gep\}).
\]
This suggests that the associated dynamic solves the stochastic differential equation
\begin{equation}\label{fdsde}
dX^i = \frac 12 \left((A_n X)^i -\, f'_0(X^i)\right) dt+ \sum_j \frac{1-e^{-\alpha_j\, 2^{-n}}}{1+e^{-\alpha_j \, 2^{-n}}} \, d\ell^{i,y_j}_t + dw_t^i
\end{equation} 
where $(\ell^{i,a}_t, t\geq 0)$ is the local time of $(X^i(t), t\geq 0)$ at $a$. Then $(X^i_t)_i$ is a vector of interacting skew Brownian motions.
\subsection{Skew Brownian motion}
Let $(X_t)_{t\geq 0}$ be the skew Brownian motion defined in \eqref{skewR} with $|\beta|<1$. Then 
\begin{lemma}\label{skewr}
The process $(X_t)_{t\geq 0}$ is associated with the Dirichlet form
\[
D(u) := \frac12 \int_\R (\dot{u})^2 \, \exp(-\alpha\un{]-\infty,0]}) \, dx
\]
in $L^2(\exp(-\alpha\un{]-\infty,0]}) \, dx)$, where $\alpha\in\R$ is defined by $\frac{1-e^{-\alpha}}{1+e^{-\alpha}}=\beta$.
\end{lemma}
\begin{proof}
The form $(D, C^1_b(\R))$ is closable in $L^2(\exp(-\alpha\un{]-\infty,0]}) \, dx)$ since it is equivalent to the standard Dirichlet forms associated with the Brownian motion.  By the same argument, the closure of $(D, C^1_b(\R))$ is regular and therefore there exists an associated Hunt process $(X_t)_{t\geq 0}$. We want now to prove that this process is a weak solution of \eqref{skewR}. The following integration by parts
formula
\[
\begin{split}
\int \varphi' \, \exp(-\alpha\un{]-\infty,0]}) \, dx = &
- (1-e^{-\alpha}) \, \varphi(0) \\ = & \ 2\, \frac{1-e^{-\alpha}}{1+e^{-\alpha}} \lim_{\varepsilon\downarrow 0} \frac1{2\varepsilon} \int_{-\gep}^\gep \varphi \, \exp(-\alpha\un{]-\infty,0]}) \, dx,
\end{split}
\]
together with the Fukushima decomposition, shows that $X_t$ is a semimartingale and that it satisfies \eqref{skewR} for quasi-every initial point $X_0=x$, i.e. for all $x$ outside a set $N$ of null capacity. However, we can in fact choose $N=\emptyset$ by noting that the transition semigroup of the skew Brownian motion 
with $-1\leq \beta \leq 1$ has an explicit Markov transition density with respect to the Lebesgue measure  (see III.1.16, VII.1.23, XII.2.16 in \cite{reyo}).
Therefore $X$ satisfies the {\it absolute continuity assumption} and we can use
\cite[Theorem 4.1.2 and formula (4.2.9)]{fot}.
 
\end{proof}

\begin{theorem}\label{skewb}
The form $ \Lambda_n$, defined in \eqref{Lambda_n}, is a regular Dirichlet form in $L^2(\pi_n)$, and the associated Markov process is a weak solution of \eqref{fdsde}. Moreover such solution is unique in law.
\end{theorem}

\begin{proof} 
As in the proof of Lemma \ref{skewr},
$ \Lambda_n$ is a regular Dirichlet form with the strong local property because it is equivalent to the Dirichlet form of a finite dimensional Ornstein-Uhlenbeck process. So by \cite{fot} there is a continuous Hunt process associated to $ \Lambda_n$. 

By the integration by parts formula \eqref{ibpfn} and the Fukushima decomposition, 
the Hunt process associated with $ \Lambda_n$ has the following property: the process $(\langle h, X_t\rangle)_{t\geq 0}$ is a semi-martingale
\begin{equation}\label{decomposition1}
\langle h, X^n_t\rangle - \langle h, X^n_0\rangle = M^{h}_t + N^{h}_t
\end{equation}
and the Revuz measure of the bounded-variation CAF $N^{h}$ is
\begin{equation}\label{Revuzmeasure}
\Sigma^{h}(dx) = \frac12 \langle A_n x-f'_0(x), h\rangle \,   \pi_n(dx) +
\sum_{i=0}^{2^n-1} h_i \, \sum_j \frac{1-e^{-\alpha_j\, 2^{-n}} }
{1+e^{-\alpha_j \, 2^{-n}}} \,  \pi_n(dx \, ; \, x(i)=y_j).
\end{equation}
Because of the structure of $\Sigma^{h}$, the process $N^h$ can be written as
\[
N^h_t = \int_0^t \frac12 \langle A_n X_s-f'_0(X_s), h\rangle \, ds + \sum_{i=0}^{2^n-1} h_i \, \sum_j \frac{1-e^{-\alpha_j\, 2^{-n}} }
{1+e^{-\alpha_j \, 2^{-n}}} \, \ell^{i,y_j}_t,
\]
where $\ell^{i,y_j}_t$ is adapted to the natural filtration of $(X_t,t \geq 0)$.
We want now to show that in fact $\ell^{i,y_j}_t$ is adapted to the natural filtration of $(X_t^i,t \geq 0)$. Since $X_t^i$ is a semimartingale, by Tanaka's formula
\begin{equation}\label{tanaka}
|X_t^i-y_j| = |X_0^i-y_j|+ \int_0^t {\rm sign}(X_s^i-y_j)\, d X_s^i + L_t^{y_j}(X^i)
\end{equation}
where $L^{y_j}(X^i)$ is the local time of $X_t^i$ at $y_j$. Since $|\langle e_i, \, \cdot \rangle-y_j|\in\cE^f_n$, then $L^{y_j}(X^i)$ is an additive functional of $X$. Now we can compute the Revuz measure of $L^{y_j}(X^i)$, using theorem 5.4.2 of \cite{fot}. With an integration by parts formula we see that for all $\varphi$ smooth enough:
\[
\begin{split}
 & \cE^f_n(|\langle e_i,\cdot \rangle-y_j|,\varphi) = \frac12\int {\rm sign}(x_i-y_j)\, \partial_i\varphi(x) \, d \pi_n \\  & = - \frac12\int{\rm sign}(x_i-y_j)\, \left((A_nx)^i-f'_0(x_i)\right) \, \varphi(x) \, d \pi_n 
 - \int \varphi(x) \,  \pi_n(dx; x(i)=y_j).
 \end{split}
 \]
 By comparison with \eqref{tanaka}, we see that $ \pi_n(dx; x(i)=y_j)$ is the Revuz measure of $t\mapsto L_t^{y_j}(X^i)$ and therefore by \eqref{Revuzmeasure} the processes $(L_t^{y_j}(X^i), t\geq 0)$ and $(\ell^{i,y_j}_t, t\geq 0)$ are equal up to a multiplicative constant.

We want now to prove uniqueness in law for \eqref{fdsde}. We define the exponential martingale
\[
M_t :=\exp\left( -\int_0^t \frac12\langle A_nX_s-f'_0(X_s), dw_s\rangle -\frac18
\int_0^t \| A_nX_s-f'_0(X_s)\|^2 ds \right).
\]
Then under the probability measure $M_T\cdot\bbP_x$, by the Girsanov theorem the canonical process is a solution in law of
\[
dX^i = \sum_j \frac{1-e^{-\alpha_j\, 2^{-n}}}{1+e^{-\alpha_j \, 2^{-n}}} \, d\ell^{i,y_j}_t + d\hat w_t^i, \qquad t\in[0,T],
\]
where the Brownian motions $(\hat w_t^i, t\geq 0)_i$ are independent; therefore we have reduced to an independent vector of skew-Brownian motions and uniqueness in law holds for such processes by the pathwise uniqueness proved by Harrison and Shepp in \cite{hash}.

Moreover, by the property recalled in the proof of Lemma \ref{skewr}, the transition semigroup of the skew-Brownian motion satisfies the absolute continuity condition and therefore all the above statements are true for all initial conditions.
\end{proof}
\subsection{Convergence of the Hilbert spaces}\label{Convergence of the Hilbert spaces}
\begin{prop}\label{convHfin}
The sequence of Hilbert spaces $(L^2( \pi_n))_n$ converges to $L^2( \nu)$ in the sense of Definition \ref{kuwshi}.
\end{prop}
\begin{proof}
According to Definition \ref{kuwshi}, we have first to define a map $\Phi_n:L^2( \nu)\mapsto L^2( \pi_n)$.
We consider now the Borel $\sigma$-field $\cB$ on $L^2(0,1)$, completed with all $\mu$-null sets (we use the same notation for the completed $\sigma$-field).

Setting $\bar\beta:=P_n\beta$, let us introduce the filtration $\cF_n:=\sigma( \bar{\beta}_{i2^{-n}}, \, i=1,\ldots,2^{n})$ and the linear map $\Phi_n:L^2(\mu)\mapsto L^2(\mu_n)$ defined as follows: $ \Phi_n(\varphi)=\varphi_n$, where
\[
\varphi_n( \bar{\beta}_{i2^{-n}}, \, i=1,\ldots,2^{-n}) = \E( \varphi(\beta) \, | \, \cF_n).
\]
Then $\varphi_n$ is well defined $\mu_n$-a.e. For any $\varphi\in L^2(\mu)$ the sequence $(\varphi_n)_n$ is a martingale bounded in $L^2(\mu)$, therefore converging a.s. and in $L^2(\mu)$.
Now, since $L^2(\mu)\equiv L^2( \nu)$ and $L^2(\mu_n)\equiv L^2( \pi_n)$ with equivalence of norms (uniformly in $n$),
then the map $ \Phi_n$ is still well defined and $\sup_n\| \varphi_n\|_{L^2( \pi_n)}<+\infty$ for all $\varphi\in L^2( \nu)$. We have
to prove that $\| \varphi_n \|_{L^2( \pi_n)}{\rightarrow} \| \varphi\|_{L^2( \nu)}$ as $n\to+\infty$.

We first prove that $F_n(\bar{\beta}^n)$ converges a.s. to $F(\beta)$, where $\beta^n:=\bar{\beta}_{\lfloor r2^{n}\rfloor}$, $r\in[0,1]$.
We have that
\[
F_n(\beta^n)=2^n\sum_{i=1}^{2^n-1} f_n(\beta_{i2^{-n}}) = \int_0^1 f_n(\beta_{\lfloor r2^{n}\rfloor}) \, dr.
\]
Now  by dominated convergence it is enough to prove that a.s. $f_n(\beta^n_r)\underset{n\to +\infty}{\rightarrow} f(\beta_r)$ for a.e. $r\in[0,1]$. By \eqref{f}, $f$ is continuous outside the finite set $\Delta_f=\{y_j\}$. Moreover $(f_n)_n$ is uniformly bounded and $f_n \underset{n\to +\infty}{\rightarrow} f$ as $n\to+\infty$ locally uniformly in $\R\setminus \Delta_f$.
For all $a\in\R$,
a.s. $\{r\in[0,1]: \beta_r=a\}$ is a compact set with zero Lebesgue measure and
therefore a.s. $U:=\{r\in[0,1]: \beta_r\in \Delta_f\}$ also has zero Lebesgue measure. Therefore
for all $r\in[0,1]\setminus U$,  $f_n(\beta^n_r)\underset{n\to +\infty}{\rightarrow} f(\beta_r)$ and
by dominated convergence $F_n(\beta^n)$ converges a.s. to $F(\beta)$.
In particular, by dominated convergence ${Z_n}=\mu_n(e^{-F_n})=\E(e^{-F_n(\bar{\beta}^n)})$ converges to $Z=\E(e^{-F(\beta)})$.

Now, let us prove that $\| \varphi_n \|_{L^2( \pi_n)}\to \| \varphi\|_{L^2( \nu)}$. Since $Z_n\underset{n\to +\infty}{\rightarrow} Z$, we have to prove that
\[
\E\left( \varphi_n^2(\bar{\beta}^n) \, e^{-F_n(\bar{\beta}^n)} \right) \underset{n\to +\infty}{\rightarrow} \E\left( \varphi^2(\beta) \, e^{-F(\beta)} \right).
\]
We have shown above that $\varphi_n(\bar{\beta}^n)$ converges to $\varphi(\beta)$ in $L^2$. Therefore $(\varphi_n^2(\beta^n))_n$ is
uniformly integrable and so is also $(\varphi_n^2(\bar{\beta}^n) \, e^{-F_n(\bar{\beta}^n)})_n$, since $(e^{-F_n(\bar{\beta}^n)})_n$ is bounded in $L^\infty$. We can then conclude since a u.i. sequence converging a.s. converges in $L^1$.
\end{proof}

\subsection{Mosco convergence}
We want now to prove that $\Lambda^n$ Mosco converges to $\cE$. In \cite[Thm. 3.5]{avr}, Andres and von Renesse have proved that Theorem \ref{moscoks} still holds if one replaces the condition {\it Mosco II} with the following condition {\it Mosco II'}.
\begin{definition}[{\it Mosco II'} ] There is a core $\cK\subset\cD(\cE)$ such that for any $x\in K$ there exists a sequence $x_n\in \cD(\Lambda^n)$ converging strongly to $x$ and such that $\cE(x,x)=\lim\limits_{n\to +\infty} \Lambda^n(x_n,x_n)$.
\label{stronglycv}
\end{definition}
\begin{theorem}[Andres and von Renesse \cite{avr}]\label{moscoks'}
The conditions Mosco I and Mosco II' are equivalent to the Mosco convergence.
\end{theorem}

\begin{theorem}\label{Mosco}
The Dirichlet form $\Lambda_n$ Mosco-converges to $\Lambda$ as $n\to+\infty$.
\end{theorem}

\begin{lemma}\label{comp2}
Let $u_n\in L^2(\pi_n)$ be a sequence which converges weakly to $u\in L^2(\nu)$, and such that $\liminf_{n}\Lambda^n(u_n,u_n)<+\infty$, then there is a subsequence of $(u_n\circ P_n)_n$ converging to $u$ in $L^2(\nu)$.
\end{lemma}
\begin{proof}
By passing to a subsequence, we can suppose that $\limsup_{n}\Lambda^n(u_n,u_n)<+\infty$.
By \eqref{E^f}, we have that $\cE^0(u_n\circ P_n,u_n\circ P_n)\leq C\Lambda^n(u_n,u_n)$, for some constant $C>0$,
and therefore $\limsup_{n} \cE^0(u_n\circ P_n,u_n\circ P_n)<+\infty$. By Proposition \ref{compact}, the inclusion $D(\cE^0)\subseteq L^2(\nu)$ is compact, so that we can extract a subsequence $v_{n_k}:=u_{n_k}\circ P_{n_k}$ converging in $L^2(\nu)$. This subsequence $u_{n_k}\in L^2(\pi_{n_k})$ converges strongly to $u\in L^2(\nu)$, since $\Phi_n(u_n\circ P_n) = u_n$, by the definition of $\Phi_n$ given in the proof of Proposition \ref{convHfin}. \end{proof}
\begin{proof}[Proof of Theorem \ref{Mosco}]
Let us consider the following regularization of $f$: we fix a function $\rho:\R\mapsto\R$ such that $\rho(x)=1$ for all $x\leq 0$, $\rho(x)=0$ for all $x\geq 1$, $\rho$ is monotone non-increasing and twice continuously differentiable on $\R$ with $0\leq \rho'\leq 1$; then we set
\[
f_n(y) = f_0(y) + \sum_{j=1}^k \alpha_j \, \rho(n(y-y_j)+\un{(\alpha_j<0)}), \qquad y\in\R.
\]
Notice that $f_n\downarrow f$ pointwise as $n\uparrow+\infty$.
Now we define the measure
\[
\tilde \pi_n(dx) = \frac 1{Z_n} \exp(-F_n(x))\, \mu_n(dx)
= \frac 1{Z_n} \exp\left(-\frac1 {2^n}\sum_{i=1}^{2^n-1} f_n(x(i2^{-n}))\right)\, \mu_n(dx);
\]
note that $\tilde \pi_n$ is not normalized to be a probability measure, in fact $\tilde \pi_n\leq \pi_n$ since $f_n\geq f$. We also define the Dirichlet form
\[
\tilde \Lambda^n(\varphi,\psi) := \frac12\int \langle\nabla\varphi,\nabla\psi\rangle \, d\tilde  \pi_n, \qquad \forall \, \varphi,\psi\in D(\Lambda^n).
\]
The form $\tilde \Lambda^n$ is clearly equivalent to $\Lambda^n$ on $D(\Lambda^n)$. Moreover $\tilde \Lambda^n(u,u)\leq \Lambda^n(u,u)$ for all $u\in D(\Lambda^n)$.

Let us show first condition Mosco II'. For $v\in \cK:={\rm Exp}_A(H)$, we have that
\[
v(w) = \sum_{m=1}^k\lambda_k \exp(i\langle w,h_m\rangle)
\]
and we can suppose that $v\ne 0$. We set $v_n:=v_{|H_n}$. Then it is easy to see that $v_n$ converges strongly to $v$; indeed, setting $\tilde v_n:=v\circ P_n$, we have $\Phi_m(\tilde v_n)=v_n$ for $m\geq n$ by construction; therefore
\[
\| \Phi_m(\tilde v_n)-v_m\|_{L^2(\pi_m)} = \| v_n-v_m\|_{L^2(\pi_m)} \leq C \| v\circ P_n-v\circ P_m\|_{L^2(\mu)},
\]
which tends to 0 as $m\to+\infty$ and then $n\to+\infty$. Moreover 
\[
\Lambda^n(v_n,v_n) = \frac 12\int \|P_n\nabla v\|^2_H \, d\pi_n \to \cE(v,v),
\]
so that Mosco II' holds. 

Let us prove now Mosco I. Let $u_n\in L^2(\pi_n)$ be a sequence converging weakly to $u\in L^2(\nu)$; we can suppose that $u\in\cD(\cE)$ and that $\liminf\limits_{n} \Lambda^n(u_n,u_n) < +\infty$;
then by lemma \ref{comp2}, up to passing a subsequence, we can suppose that $u_{n} \to u$ strongly. 

Since $\tilde \Lambda^n\leq \Lambda^n$, we have 
\[
\liminf\limits_{n\to \infty} \Lambda^n(u_n,u_n) \geq \liminf\limits_{n\to \infty} \tilde \Lambda^n(u_n,u_n).
\]
Now for any $v_n\in D(\Lambda^n)$ 
\begin{equation}
\tilde \Lambda^n(u_n,u_n)  \geq  
\frac{\big(\tilde \Lambda^n(u_n,v_n)\big)^2}{\tilde \Lambda^n(v_n,v_n)}.
\end{equation}
Suppose that $v\ne 0$ and $v\in{\rm Exp}_A(H)$ is a linear combination of exponential functions. We set $v_n:=v_{|H_n}$. Then arguing as above we have $\tilde \Lambda^n(v_n,v_n)\to \cE(v,v)$. Now we prove that $\tilde \Lambda^n(u_n,v_n)\to \cE(u,v)$. By linearity, we can suppose that
$v = \exp(i\langle \cdot,h\rangle)$. Integrating by parts we see that
\[
\begin{split}
2\, \tilde \Lambda^n(u_n,v_n) & = -i\int u_n(x) \, v_n(x) \, \langle A_n x-\, f'_n(x), P_nh\rangle \,  \pi_n(dx).
\end{split}
\]
The claim follows if we prove that
\[
\int u_n(x) \, v_n(x) \, \langle n \rho'(n(x-y)), P_nh\rangle \,  \pi_n(dx) \to
\int u(x) \, v(x) \, \langle \ell^y_{\cdot},h\rangle \,  \nu(dx).
\]
Note that, with the notation $\beta^n=P_n\beta$,
\[
\int \varphi(x) \, \langle n \rho'(n(x-y)), h\rangle \,  \pi_n(dx) =
\E( \varphi(\beta^n) \, \langle n \rho'(n(\beta^n-y)), h\rangle ).
\]
Now 
\[
|\langle n \rho'(n(\beta^n-y))-n \rho'(n(\beta-y)), P_nh\rangle|
\leq n \sup_{|r-s|\leq 2^{-n}} |\beta_r-\beta_s| \, \|h\|_\infty.
\]
Moreover, if $h$ has support in $[\gep,1-\gep]$, then
\[
\begin{split}
& \left|\langle n \rho'(n(\beta-y)), h_n\rangle - \int_0^1 h_n \, d\ell^y\right|= 
\left|\int h_n(r) \left(\int n \rho'(n(a-y)) \, (\ell^a - \ell^y)(dr)\, da\right)\right|
\\ & = \left|\int_\gep^{1-\gep} h_n'(r) \left(\int n \rho'(n(a-y)) \, (\ell^a(r) - \ell^y(r))\, da\right) dr\right|
\\ &  \leq \|h'\| \, \sup_{|a-y|\leq 1/n} \sup_{r\in[\gep,1-\gep]} |\ell^a(r) - \ell^y(r)|.
\end{split}
\]
We want now to show that these quantities converge to $0$ in $L^2$ as $n\to+\infty$.
Indeed, since $(\beta_{1-r}, r\in[0,1])$ has the same law as $(\beta_{r}, r\in[0,1])$, we can write
\[
\begin{split}
&\E\left(\sup_{|r-s|\leq 2^{-n}} |\beta_r-\beta_s|^2\right) \leq 2\, \E\left(\sup_{|r-s|\leq 2^{-n},
r,s\leq \frac34} |\beta_r-\beta_s|^2\right) 
\\ & = 2\, \E\left(\sup_{|r-s|\leq 2^{-n}, \,
r,s\leq \frac34} |B_r-B_s|^2 \, \frac{p_{1/4}(B_{3/4})}{p_{1}(0)}\right) \leq
4\, \E\left(\sup_{|r-s|\leq 2^{-n}, \, r,s\leq \frac34} |B_r-B_s|^2 \right)
\\ & \leq C (2^{-n})^{1/2}
\end{split}
\]
by Kolmogorov's continuity criterion for the standard Brownian motion $(B_r)_{r\geq 0}$.
For the other term, we also reduce to a known result on the local time $(\ell^a_t)_{a\in\R,t\geq 0}$ of Brownian motion:
\[
\begin{split}
&\E\left(\sup_{|a-y|\leq 1/n} \sup_{r\in[\gep,1-\gep]} |\ell^a(r) - \ell^y(r)|^2\right) 
\\ & = \E\left(\sup_{|a-y|\leq 1/n} \sup_{r\in[\gep,1-\gep]} |\ell^a(r) - \ell^y(r)|^2 \, \frac{p_{\gep}(B_{1-\gep})}{p_{1}(0)}\right) 
\\ & \leq \gep^{-1/2} \, \E\left(\sup_{|a-y|\leq 1/n} \sup_{r\in[\gep,1-\gep]} |\ell^a(r) - \ell^y(r)|^2\right) 
\leq C (1/n)^{1/2},
\end{split}
\]
see  \cite{reyo} p.225-226. It only remains to prove that
\begin{equation}\label{wal}
\lim_n \int u_n \, v_n \, \langle A_n x-\, f'_0(x), P_nh\rangle \,  \pi_n(dx)
= \int u \,  \, (\langle x, Ah\rangle-\, \langle f'_0(x),h\rangle) \,  \nu(dx).
\end{equation}
The term containing $f'_0(x)$ gives no difficulty; as for 
$\int u_n\, v_n \, \langle  \cdot,A^nP_n h\rangle\, d \pi_n$, we have
\[
\int u_n\, v_n \, \langle  \cdot,A^nP_n h\rangle\, d \pi_n = \frac1{Z_n} \int u_n\, v_n \, \langle  \cdot,A^nP_n h\rangle\, e^{-F_n} \, d\mu_n.
\]
Now, notice that by an integration by part formula, we have for all $g\in C^1_b(H)$
\[
\int g\circ P_n \,  \langle  \cdot,A^nP_n h\rangle\, d\mu =
\int g\,  \langle  \cdot,A^nP_n h\rangle\, d\mu_n = - \int \partial_{P_n h} g \, d\mu_n
\]
and, again by an integration by parts formula, 
\[
- \lim_{n\to+\infty} \int \partial_{P_n h} g \, d\mu_n = - \int \partial_h g \, d\mu = 
\int g\,  \langle  \cdot,A h\rangle\, d\mu.
\]
Moreover 
\[
\int \langle  \cdot,A^nP_nh\rangle^2d\mu = 
\int \langle  \cdot,A^nP_nh\rangle^2d\mu_n = \|P_nh\|^2 \leq \|h\|^2.
\]
%
%
Therefore, the linear functional $L^2(\mu)\ni g\mapsto \int g\circ P_n \,  \langle \cdot,A^nP_n h\rangle\, d\mu$
is uniformly bounded in $n$ and converges on  $C^1_b(H)$, a dense subset in $L^2(\mu)$. By a density argument, this sequence of functionals converges weakly in $L^2(\mu)$.

We recall now that $L^2(\pi_n)\ni u_n$ converges strongly to $u\in L^2(\nu)$.
We want to show that $(u_nv_ne^{-F_n})\circ P_n\to uve^{-F}$ in $L^2(\mu)$. 
Indeed by lemma \ref{comp2}, from any subsequence of $(u_n\circ P_n)_n$ we can extract a sub-subsequence converging to $u$ in $L^2(\nu)$ and $\nu$-almost surely. On the other hand $(v_ne^{-F_n})\circ P_n$ converges pointwise to $ve^{-F}$ and $((v_ne^{-F_n})\circ P_n)_n$ is uniformly bounded, so we conclude with the dominated convergence theorem. Therefore, we obtain that 
\[
\lim_n \int u_nv_ne^{-F_n}\,  \langle  \cdot,A^nP_n h\rangle\, d\mu_n = \int uve^{-F}\,  \langle  \cdot,A h\rangle\, d\mu,
\]
and \eqref{wal} is proved.

Finally we prove that if $\liminf\limits_{n}\Lambda^n(u_n,u_n)<+\infty$, then $u\in\cD(\cE)$. Indeed for all $u_n\in\cD(\Lambda^n)$ we have $u_n\circ P_n\in \cD(\cE)$, moreover $(u_n)_n$ converges weakly to $u$ then $(u_n\circ P_n)_n$ converges weakly to $u$ in $L^2(\nu)$; then, as at the end of the proof of Proposition \ref{MI}, by the compact injection of $\cD(\cE)$ in $L^2(\nu)$ we have that $u\in\cD(\cE)$, which ends the proof.
\end{proof}
\subsection{Convergence in law of stationary processes}
We denote now by $(\bbQ_{\pi_n}^n)_n$ the law of the stationary solution of equation \eqref{fdsde} started with initial law $\pi_n$.
We want to prove a convergence result for $(\bbQ_{\pi_n}^n)_n$ to $\bbP_\nu$, the stationary solution to equation \eqref{1}. We define the space $H^{-1}(0,1)$ as the completion of $L^2(0,1)$ with respect to the Hilbertian norm
\[
\|x\|^2_{H^{-1}(0,1)} := \int_0^1 d\theta \, \langle x,\un{[0,\theta]}\rangle_{L^2(0,1)}^2,
\]
and the linear isometry $J:H^{-1}(0,1)\mapsto L^2(0,1)$ given by the closure of 
\[
H^{-1}(0,1)\subset L^2(0,1)\ni x\mapsto 
Jx:=\langle x,\un{[0,\cdot]}\rangle_{L^2(0,1)}.
\]
\begin{lemma} The sequence $\bbQ_{\pi_n}^n$ converges weakly to $\bbP_\nu$ in $C([0,T]; H^{-1}(0,1))$.
\end{lemma}
\begin{proof} We define $\bbS_n:=\bbQ_{\pi_n}^n\circ J^{-1}$, i.e. the law of $(JX_t^n)_{t\geq 0}$, where $X_t^n$ has
law $\bbQ_{\pi_n}^n$. Since $J$ maps $L^2(0,1)$ continuously into $H^1(0,1)$, we obtain that ${\pi_n}^n\circ J^{-1}$ satisfies condition \eqref{cater} below. Therefore by Lemma \ref{apriori} below, $(\bbS_n)_n$ is tight in $C([0,T]\times[0,1])$ and therefore $(\bbQ_{\pi_n}^n)_n$ is tight in $C([0,T]; H^{-1}(0,1))$.

Let us now prove convergence of finite dimensional distributions. As in the proof of Proposition \ref{convergence}, let $f\in C_b(H^m)$ of the form $f(x_1, ...,x_m) = f_1(x_1)\cdots f_m(x_m)$. By the Markov property, it is enough to prove that
\[
\begin{split}
& P^n_{t_1}(f_1\cdot P^n_{t_2-t_1}(f_2\cdot ... (f_{m-1}P^n_{t_m - t _{m-1}}f_m)...)) 
\\ & \to 
P_{t_1}(f_1\cdot P_{t_2-t_1}(f_2\cdot ... (f_{m-1}P_{t_m - t _{m-1}}f_m)...)), \qquad {\rm strongly}.
\end{split}
\]
Arguing by recurrence, we only need to prove that, if $L^2(\pi_n)\ni v_n\to v\in L^2(\nu)$ strongly, and $g\in C_b(H)$, then $L^2(\pi_n)\ni g\cdot v_n$ converges strongly to $g\cdot v\in L^2(\nu)$. We have
\[
\begin{split}
& \| \Phi_m(g\cdot \tilde{v}_n) - g\cdot v_m\|_{L^2(\pi_m)} \\ & \leq \| \Phi_m(g\cdot \tilde{v}_n - g\circ P_m\cdot \tilde{v}_n)\|_{L^2(\pi_m)} + \| g\cdot (\Phi_m(\tilde{v}_n) - v_m)\|_{L^2(\pi_m)}.
\end{split}
\]
Recalling that $\Phi_m$ is defined in terms of a conditional expectation, see the proof of Proposition \ref{convHfin}, we obtain
\[
\limsup_m \| \Phi_m(g\cdot \tilde{v}_n - g\circ P_m\cdot \tilde{v}_n)\|_{L^2(\pi_m)} \leq \limsup_m C\| (g - g\circ P_m) \tilde{v}_n\|_{L^2(\nu)} = 0,
\]
since the conditional expectation is a contraction in $L^2(\mu)$ and $g\circ P_m$ converges almost surely to $g$ if $m\to+\infty$. Moreover 
\[
\lim_n \limsup_m \| g\cdot (\Phi_m(\tilde{v}_n) - v_m)\|_{L^2(\pi_m)} \leq \|g\|_{\infty}
\lim_n \limsup_m \| \Phi_m(\tilde{v}_n) - v_m\|_{L^2(\pi_m)} = 0
\]
by assumption. Therefore $L^2(\pi_n)\ni g\cdot v_n$ converges strongly to $g\cdot v\in L^2(\nu)$ and
we obtain the convergence in law of the finite dimensional laws.
\end{proof}

\section{A priori estimate}
We prove in this section an estimate which has been used above to prove tightness properties in $C([0,T]\times[0,1])$.
We consider here a probability measure $\gamma$ on $H$  and
Dirichlet form $(\bbD, D(\bbD))$ in $L^2(\gamma)$ such that $C^1_b(H)$ is a core of $ \bbD$ and
\[
\bbD(u,v) = \frac12 \int \langle\nabla u,\nabla v\rangle \, d\gamma, \qquad \forall \, u,v\in C^1_b(H).
\]
Let us define for $\eta\in\,]0,1[$ and $r\geq 1$ the norm $\|\cdot\|_{W^{\eta,r}(0,1)}$, given by 
\[
\|x\|_{W^{\eta,r}(0,1)}^r = \int_0^1 |x_s|^r ds + \int_0^1\int_0^1 \dfrac{|x_s - x_t|^r }{|s-t|^{r\eta+1}} \, dt \, ds.
\]
Then we have the following
\begin{lemma}\label{apriori}
Let $(X_t)_{t\geq 0}$ be the stationary Markov process associated with $\bbD$, i.e. such that the law of $X_0$ is $\gamma$. Suppose that there exist $\eta\in\,]0,1[$, $\zeta>0$ and $p>1$ such that
\[
\zeta>\frac1{1+\frac23\eta}, \qquad p>\max\left\{\frac2{1-\zeta}, \frac1{\eta-\frac32\frac{1-\zeta}\zeta}
\right\},
\]
and
\begin{equation}\label{cater}
\int_H \| x \|_{W^{\eta,p}(0,1)} ^p \, \gamma(dx) = C_{\eta,p} <+\infty.
\end{equation}
Then there exist $\theta\in\,]0,1[$, $\xi>1$ and $K>0$, all depending only on $(\eta,\zeta,p)$, such that
\[
{\mathbb E}\left[\left\| X_t-
X_s\right\|^p_{C^{\theta}([0,1])} \right]
 \, \leq \, K \, |t-s|^{\xi} .
\]
\end{lemma}
\begin{proof}
\medskip\noindent We follow the proof of Lemma 5.2 in \cite{deza}.
We introduce first the space $H^{-1}(0,1)$,
completion of $L^2(0,1)$ w.r.t. the norm:
\[
\|f\|_{-1}^2 \, := \, \sum_{k=1}^\infty k^{-2} \,
|\langle f, e_k \rangle_{L^{2}(0,1)}|^2
\]
where $e_k(r):={\sqrt 2}\sin(\pi kr)$, $r\in[0,1]$, $k\geq 1$, are the eigenvectors of the second derivative with
homogeneous Dirichlet boundary conditions at $\{0,1\}$. Recall that $L^{2}(0,1)=H$, in
our notation.
We denote by $\kappa$ the Hilbert-Schmidt norm of
the inclusion $H \to H^{-1}(0,1)$, which by definition is equal in our case to
\[
\kappa = \sum_{k\geq 1} k^{-2}<+\infty.
\]
We claim that for all $p>1$ there exists $C_p\in(0,\infty)$, depending only on $p$, such that
\begin{equation}\label{estim}
\left( {\mathbb E}\left[
\left\|X_t- X_s\right\|^p_{H^{-1}(0,1)} \right]
\right)^{\frac1p} \, \leq \, C_p \,\kappa\, |t-s|^{\frac12}, \qquad
t,s\geq 0.
\end{equation}
To prove (\ref{estim}),
we fix $T>0$ and use the Lyons-Zheng
decomposition, see e.g. \cite[Th. 5.7.1]{fot}, to write for
$t\in[0,T]$ and $h\in H$:
\[
\langle h,X_t - X_0\rangle_H
\, = \, \frac 12 \,
M_t \, - \, \frac 12 \, (N_T  - N_{T-t}),
\]
where $M$, respectively $N$, is a martingale w.r.t. the natural
filtration of $X$, respectively of $(X_{T-t},
\ t\in[0,T])$.
Moreover, the quadratic variations are both equal to:
$\langle M\rangle_t =\langle N\rangle_t =t \cdot \|h\|^2_H$. By the
Burkholder-Davis-Gundy inequality we can find $c_p\in(0,\infty)$
for all $p>1$ such that: $\left({\mathbb E}\left[|\langle X_t-
X_s, e_k \rangle|^p \right]\right)^{\frac1p} \leq c_p \,|t-s|^{\frac12}$,
$t,s\in[0,T]$, and therefore
\[
\begin{split}
& \left({\mathbb E}\left[\left\|X_t-
X_s\right\|^p_{H^{-1}(0,1)} \right]
\right)^{\frac1p}  \leq  \sum_{k\geq 1} k^{-2}\left({\mathbb E}\left[|\langle X_t-
X_s, e_k \rangle|^p \right]
\right)^{\frac1p} \\
  & \leq  c_p \sum_{k\geq 1} k^{-2} |t-s|^{\frac12} \|e_k \|^{2}_{L^{2}(0,1)}
  \leq  c_p \, \kappa \,|t-s|^{\frac12}, \quad t,s\in[0,T],
\end{split}
\]
and (\ref{estim}) is proved.
By stationarity
\begin{align}\label{estim2}
\nonumber
& \left( {\mathbb E}\left[\left\|X_t-
 X_s\right\|^p_{W^{\eta,p}(0,1)}
\right]\right)^{\frac1p}
\leq  \left( {\mathbb E}\left[\left\|
 X_t\right\|^p_{W^{\eta,p}(0,1)} \right]
\right)^{\frac1p}+\left( {\mathbb E}\left[\left\|
 X_s\right\|^p_{W^{\eta,p}(0,1)} \right]
\right)^{\frac1p}
\\ \nonumber \\  & \qquad
= 2\left( \int_H \| x\|^p_{W^{\eta,p}(0,1)} \, d\gamma \right)^{\frac1p}
\, = \, 2\, (C_{\eta,p})^{1/p}.
\end{align}
By the assumption on $\zeta$ and $p$ it follows that $\alpha:=\zeta\eta-(1-\zeta)>0$ and
\[
\frac{p}2 \, (1-\zeta)>1, \qquad \frac 1d:=\zeta \frac 1p + (1-\zeta)\frac12
<\alpha.
\]
Then by interpolation, see \cite[Chapter 7]{af},
\begin{eqnarray*}
& &
\left( {\mathbb E}\left[\left\| X_t-
X_s\right\|^p_{W^{\ga,d}(0,1)} \right]
\right)^{\frac1p}\leq
\\ \\ & & \leq \,
\left( {\mathbb E}\left[\left\| X_t-
X_s\right\|^p_{W^{\eta,p}(0,1)} \right]
\right)^{\frac \zeta p}\left( {\mathbb E}\left[\left\|
X_t- X_s\right\|^p_{H^{-1}(0,1)}
\right] \right)^{\frac{1-\zeta}p}.
\end{eqnarray*}
Since $\alpha d>1$, there exists
$\theta>0$ such that $(\ga-\theta)d>1$. By the Sobolev embedding,
$W^{\ga,d}(0,1)\subset C^{\theta}([0,1])$ with continuous embedding. 
Then we find that
\[
{\mathbb E}\left[\left\| X_t-X_s\right\|^p_{C^{\theta}([0,1])} \right] \, \leq \, K \, |t-s|^{\xi} 
\]
with $\xi:=\frac{p}2 \, (1-\zeta)>1$ and $K$ a constant depending only on $(\eta,\zeta,p)$.
\end{proof}


\begin{thebibliography}{abc99xys}


\bibitem{af} R. A. Adams, J.J.F. Fournier, {\it
Sobolev Spaces}, Secon Edition,
Academic Press, Elsivier.

\bibitem{asz} L. Ambrosio, G. Savar\'e, L. Zambotti (2009),
{\it Existence and Stability for Fokker-Planck equations
with log-concave reference measure}, Probability Theory and Related Fields,
 145 {\bf 3}, Page 517

\bibitem{avr} S. Andres, M.K. von Renesse, {\it Particle approximation of the Wasserstein diffusion}, J. Funct. Anal. 258 (2010), no. 11.

\bibitem{avr2} S. Andres, M.K. von Renesse, {\it Uniqueness and Regularity Properties for a System of Interacting Bessel Processes via the Muckenhoupt Condition},
(2011), to appear in Trans. Amer. Math. Soc.


\bibitem{sc} S. Cerrai (2001), {\it
Second Order PDE's in Finite and Infinite Dimension},
Lecture Notes in Mathematics 1762, Springer Verlag.

\bibitem{dpz2} G. Da Prato, J. Zabczyk (1996), {\it Ergodicity
for Infinite Dimensional Systems}, London Mathematical
Society Lecture Notes, n.229, Cambridge University Press.

\bibitem{dpz3} G. Da Prato, J. Zabczyk (2002),
{\it Second order partial differential equations in Hilbert spaces},
London Mathematical Society Lecture Note Series, n. 293.

\bibitem{deza} A. Debussche, L. Zambotti (2007),
{\it Conservative Cahn-Hilliard equation with reflection},
Ann. Prob., {\bf 35 (5)}, pp. 1706-1739.

\bibitem{dgz} J.-D. Deuschel, G. Giacomin and L. Zambotti,
\textit{Scaling limits of  equilibrium wetting models
in  (1+1)--dimension}, Probab. Theory Relat. Fields {\bf 132} (2005), 471--500.

\bibitem{dunsch} N.J. Dunford, J.T. Schwartz (1988),
{\it Linear Operators: Spectral Theory: Part II}, Wiley classic Library.

\bibitem{ek} S. N. Ethier, T. G. Kurtz (2005),
{\it Markov Processes: Characterization And Convergence}, 2nd Revised edition, Wiley Series in Probability and Statistics

\bibitem{fot} M. Fukushima, Y. Oshima, M. Takeda (1994),
{\it Dirichlet Forms and Symmetric Markov Processes}, Walter
de Gruyter, Berlin-New York.

\bibitem{funaki} Funaki, Tadahisa {\it Stochastic interface models}, Lectures on probability theory and statistics, 103-274, Lecture Notes in Math., 1869, Springer, Berlin, 2005.

\bibitem{fuol} T. Funaki, S. Olla (2001), {\it Fluctuations
for $\nabla\phi$ interface model on a wall},
Stoch. Proc. and Appl, {\bf 94}, no. 1, 1--27.

\bibitem{hash} Harrison, J. M.; Shepp, L. A. {\it On skew Brownian motion}. Ann. Probab. 9 (1981), no. 2, 309--313.

\bibitem{koles} A.V. Kolesnikov (2006), {\it Mosco convergence of Dirichlet forms in infinite dimensions with changing reference measures},  J. Funct. Anal.  230 {\bf 2}, 382--418.

\bibitem{kuwshi} K. Kuwae, T. Shioya (2003), {\it Convergence of spectral structures: a functional analytic theory and its applications to spectral geometry}, Comm. Anal. Geom. 11 (4) 599--673.


\bibitem{maro} Z. M. Ma, M. R\"ockner  (1992), {\it
Introduction to the Theory of (Non Symmetric) Dirichlet Forms},
Springer-Verlag, Berlin/Heidelberg/New York.


\bibitem{mosco} U. Mosco, {\it Composite media and asymptotic Dirichlet forms}, J. Funct. Anal. 123 (2) (1994) 368-421.


\bibitem{reyo}  Revuz, D., and Yor, M. (1991),
{\it Continuous Martingales and Brownian Motion}, Springer Verlag.

\bibitem{strook} {\sc D.W. Stroock, S.R.S. Varadhan (1997),} {\em Multidimensional diffusion
processes.} Springer Verlag, second~ed.

\end{thebibliography}
 \end{document}